\documentclass[10pt]{amsart}
\usepackage{amsmath,amssymb,amsfonts,eucal,graphicx,latexsym,url}
\usepackage[left=2cm,right=2cm,top=2cm,bottom=2cm]{geometry}

\newtheorem{theorem}{Theorem}
\newtheorem{thm}[theorem]{Theorem}

\newtheorem{algorithm}[theorem]{Algorithm}

\newtheorem{conjecture}[theorem]{Conjecture}
\newtheorem{corollary}[theorem]{Corollary}

\newtheorem{definition}[theorem]{Definition}
\newtheorem{example}[theorem]{Example}
\newtheorem{exception}[theorem]{Exception}

\newtheorem{lemma}[theorem]{Lemma}

\newtheorem{remark}[theorem]{Remark}

\newtheorem{alg}[theorem]{Algorithm}

\let\Im\relax
\DeclareMathOperator{\Im}{Im}
\let\Re\relax
\DeclareMathOperator{\Re}{Re}

\newcommand{\h}{\mathbb{H}}
\newcommand{\N}{\mathbb{N}}
\newcommand{\R}{\mathbb{R}}
\newcommand{\Z}{\mathbb{Z}}
\newcommand{\eps}{\varepsilon}
\DeclareMathOperator{\SL}{SL}
\DeclareMathOperator{\PSL}{PSL}
\DeclareMathOperator{\vol}{Vol}
\DeclareMathOperator{\res}{Res}
\newcommand{\F}{\mathcal{F}}
\newcommand{\A}{\mathcal{A}}
\newcommand{\M}{\mathcal{M}}

\begin{document}
\title[Distribution of eigenvalues]{On the distribution of eigenvalues of Maass
  forms on certain moonshine groups}
\author[J.~Jorgenson]{Jay Jorgenson}
\address{Department of Mathematics, The City College of New York, Convent
  Avenue at 138th Street, New York, NY 10031 USA, e-mail:
  jjorgenson@mindspring.com}
\author[L.~Smajlovi\'c]{Lejla Smajlovi\'c}
\address{Department of Mathematics, University of Sarajevo, Zmaja od Bosne 35,
  71\,000 Sarajevo, Bosnia and Herzegovina, e-mail: lejlas@pmf.unsa.ba}
\author[H.~Then]{Holger Then}
\address{Department of Mathematics, University of Bristol, University Walk,
  Bristol, BS8 1TW, United Kingdom, e-mail: holger.then@bristol.ac.uk}

\begin{abstract}
  In this paper we study, both analytically and numerically, questions
  involving the distribution of eigenvalues of Maass forms on the moonshine
  groups $\Gamma_0(N)^+$, where $N>1$ is a square-free integer.
  After we prove that $\Gamma_0(N)^+$ has one cusp, we compute
  the constant term of the associated non-holomorphic Eisenstein series.
  We then derive an ``average'' Weyl's law for the distribution of eigenvalues
  of Maass forms, from which we prove the ``classical'' Weyl's law as a
  special case.
  The groups corresponding to $N=5$ and $N=6$ have the same signature;
  however, our analysis shows that, asymptotically, there are infinitely more
  cusp forms for $\Gamma_0(5)^+$ than for $\Gamma_0(6)^+$.
  We view this result as being consistent with the Phillips-Sarnak philosophy
  since we have shown, unconditionally, the existence of two groups which have
  different Weyl's laws.
  In addition, we employ Hejhal's algorithm, together with recently developed
  refinements from \cite{The12}, and numerically determine the first $3557$ of
  $\Gamma_0(5)^+$ and the first $12474$ eigenvalues of $\Gamma_0(6)^+$.
  With this information, we empirically verify some conjectured distributional
  properties of the eigenvalues.
\end{abstract}

\thanks{J.~J.\ acknowledges grant support from NSF and PSC-CUNY grants, and
  H.~T.\ acknowledges support from EPSRC grant EP/H005188/1.}
\date{10 September 2012}

\maketitle

\section{Introduction}

Let $\{p_{i}\}$, with $i=1,\ldots,r$, be a set of distinct primes, so then
$N=p_1\cdots p_r$ is a square-free, non-negative integer.
The subset of $\SL(2,\R)$, defined by
\begin{align*}
  \Gamma_0(N)^+:=\left\{ e^{-1/2}\begin{pmatrix}a&b\\c&d\end{pmatrix}\in
    \SL(2,\R): \quad ad-bc=e, \quad a,b,c,d,e\in\Z, \quad e\mid N,\ e\mid a,
    \ e\mid d,\ N\mid c \right\}
\end{align*}
is an arithmetic subgroup of $\SL(2,\R)$.
The groups $\Gamma_0(N)^+$ were first considered by Helling \cite{Hel66}
where it was proved that if a subgroup $G\subseteq\SL(2,\R)$
is commensurable with $\SL(2,\Z)$, then there exists a square-free,
non-negative integer $N$ such that $G$ is a subgroup of $\Gamma_0(N)^+$.
We also refer to page 27 of \cite{Sh71} where the groups $\Gamma_0(N)^+$
are cited as examples of groups which are commensurable with $\SL(2,\Z)$
but non necessarily conjugate to a subgroup of $\SL(2,\Z)$.

Following the discussion in \cite{CMS04,CG97,Cum04,Cum10,Ga06b},
we employ the term ``moonshine group'' when discussing $\Gamma_0(N)^+$.
The genus zero moonshine subgroups of $\SL(2,\R)$ arise in
the ``monstrous moonshine'' conjectures of Conway and Norton, which were
later proved in the celebrated work of Borcherds.
Gannon's book \cite{Ga06b} provides an excellent discussion of the
mathematics and mathematical history of monstrous moonshine.
In particular, we refer to Conjecture 7.1.1 where the Conway-Norton conjecture
is stated, which in its original form referred to certain genus zero subgroups
{\em of Moonshine-type}.
After the work of Borcherds, the authors in \cite{CMS04} described solely in
group-theoretic terms the $171$ genus zero subgroups that appear in
mathematics of ``monstrous moonshine''.
Amongst this list are those groups of the form $\Gamma_0(N)^+$ which have
genus zero.

Our interest in the groups $\Gamma_0(N)^+$ stems from the work in \cite{JS12}.
In that article, the groups $\Gamma_0(5)^+$ and $\Gamma_0(6)^+$ were examples
of arithmetically defined topologically equivalent groups which have distinct
spectral properties.
More specifically, in \cite{JS12} the authors defined an invariant associated
to any non-compact, finite volume hyperbolic Riemann surface, where the
invariant is equal to the larger of two quantities: one coming from the
length spectrum and another associated to the determinant of the scattering
matrix.
The groups $\Gamma_0(5)^+$ and $\Gamma_0(6)^+$ have the same signature and are
arithmetically defined, yet have different values of the invariant defined in
\cite{JS12}.
As a result, the main theorem of \cite{JS12} showed that, in somewhat vague
terms, the derivative of the Selberg zeta function of one surface has more
zeros than the derivative of the Selberg zeta function of the other.
Since the spectrum of a surface is measured by the zeros of the Selberg zeta
function, the main result of \cite{JS12} can be interpreted as saying that
surfaces $\Gamma_0(5)^+$ and $\Gamma_0(6)^+$ are quite different from the
point of view of the asymptotics of spectral analysis.

In somewhat vague terms, the purpose of the present article is to investigate
the spectral properties of the Riemann surfaces associated to the groups
$\Gamma_0(N)^+$ for square-free $N$ in order to make precise the observations
made in \cite{JS12}.
In doing so, we employ the ideas from \cite{The12} which build on Hejhal's
algorithm for numerically estimating eigenvalues of the Laplacian on finite
volume, hyperbolic Riemann surfaces.
With this said, we now can describe the main results.

Let $\overline{\Gamma_0(N)^+}=\Gamma_0(N)^+/\pm I$, where $I$ is the identity
matrix and let $X_N:=\overline{\Gamma_0(N)^+}\backslash\h$ be the
corresponding two dimensional surface.
Since $\Gamma_0(N)\subseteq\Gamma_0(N)^+$, where $\Gamma_0(N)$ denotes the
classical congruence subgroup of $\SL(2,\Z)$, the surface $X_N$ has finite
volume.
As stated, we will show that for any square-free $N$, the surface $X_{N}$
has exactly one cusp; hence the signature of $\Gamma_0(N)^+$ is
$(g;m_1,\ldots,m_l;1)$ where $g$ denotes the genus of the group and $l$ is the
number of inequivalent elliptic elements of $\Gamma_0(N)^+$ with $m_i$,
$i=1,\ldots,l$ denoting the order of the corresponding elliptic element.

Maass forms on $\Gamma_0(N)^+$ are real analytic, square integrable,
eigenfunctions of the Laplacian on the surface $X_N$.
Maass forms which vanish in the cusp are called Maass cusp forms.
The hyperbolic Laplacian $-\Delta$ on $X_N$ has a discrete and continuous
spectrum; see \cite{Iw02} or \cite{Hej83}.
The discrete spectrum is denoted by the set $\{\lambda_n\}_{n\ge0}$, counted
with multiplicities; here, we have that
$0=\lambda_0<\lambda_1\le\ldots\le\lambda_{n_N-1}<1/4
\le\lambda_{n_N}\le\ldots$ and $\lambda_n\to\infty$ as $n\to\infty$.
Let $m_{1/4, N}\geq0$ denote the multiplicity of $\lambda=1/4$ as (eventual)
eigenvalue of $-\Delta$.
Maass cusp forms span the positive discrete part of the spectrum.

Let $\{r_n\}$ denote the set of all positive real numbers satisfying the
equation $1/4+r_n^2=\lambda_n$.
For $T>0$, the function $\mathcal{N}_N(T):=\mathcal{N}_N[0< r_n\le T]$
counts the number of $r_n$ such that $0< r_n\le T$, or, equivalently,
the number of eigenvalues of Maass cusp forms which lie in the interval
$(1/4, T^2+1/4]$.

For any $T>0$ and square-free $N$, which we write as $N=p_1\cdots p_r$,
define $\alpha_N(j,T):=T\log p_j-\lfloor\frac{T\log p_j}{\pi} \rfloor\pi$
where $\lfloor x\rfloor$ denotes the greatest integer less than or equal to
$x$.

The main analytical result of the paper is the following theorem:

\begin{thm}[Average Weyl's law for $\Gamma_0(N)^+$]\label{AverageWeylLaw}
  Let $(g;m_1,\ldots,m_l;1)$ be the signature of the group $\Gamma_0(N)^+$
  and let $n_N\ge1$ denote the number of small eigenvalues of the Laplacian
  $-\Delta$ on $X_N$.
  Then
  \begin{align*}
    \mathcal{N}_N(T)=\M_N(T)+S_N(T)
  \end{align*}
  where
  \begin{multline*}
    \M_N(T)=\frac{\vol(X_N)}{4\pi}T^2-\frac{2T\log T}{\pi}
    +\frac{T}{\pi}(2+\log(\pi/2N))+\sum_{i=1}^{l}\frac{1}{4m_i}
    \sum_{j=1}^{m_i-1}\frac{1}{\sin^2(\pi j/m_i)}-\frac{\vol(X_N)}{48\pi}
    -m_{1/4, N} \\
    -\frac{3}{4}-\frac{n_N}{2}+\frac{1}{2\pi}\sum_{j=1}^r \alpha_N(j,T)
    -\frac{1}{\pi}\sum_{j=1}^r\arctan\left(\left(
    \frac{\sqrt{p_j}-1}{\sqrt{p_j}+1}
    \right)^{(-1)^{\lfloor\frac{T\log p_j}{\pi}\rfloor}}
    \tan\left(\frac{\alpha_N(j,T)}{2}\right)\right)+G_N(T),
  \end{multline*}
  with
  \begin{align}\label{G_N bound}
    \vert G_N(T)\vert\le\frac{1}{2\pi}\left(
    \frac{\vol(X_N)(2\pi+1)}{2\pi^2\exp(2\pi)}
    +\sum_{i=1}^{l}\frac{m_i}{2e\pi}\sum_{j=1}^{m_i-1}
    \frac{1}{\sin(\pi j/m_i)}+\frac{5051}{900}\right)\cdot\frac{1}{T},
  \end{align}
  for all $T>1$ and
  \begin{align*}
    \int\limits_0^T S_N(t)dt=O\left(\frac{T}{\log^2T}\right)
    \ \text{ as }\ T\to\infty.
  \end{align*}
\end{thm}

The word ``average'' in the title of our main theorem relates to the form of
the error term in the Weyl's law.
An average Weyl's law is of importance when it comes to the numerical
computation of Maass forms; see \cite{The12} and references therein.
In particular, when computing Maass cusp forms numerically, there is always
the risk that some solutions get overlooked.
By comparing a numerically found list of eigenvalues of Maass cusp forms
with average Weyl's law, one can easily determine the number of solutions
which have been overlooked.
We refer to \cite{The12} for a detailed discussion of this point.

An immediate consequence of Theorem \ref{AverageWeylLaw} and its proof is the
following corollary.
\begin{corollary}[Classical Weyl's law for $\Gamma_0(N)^+$]
  \label{ClassicalWeylLaw}
  \begin{align*}
    \mathcal{N}_N(T)=\frac{\vol(X_N)}{4\pi}T^2-\frac{2T\log T}{\pi}
    +\frac{T}{\pi}(2+\log(\pi/2N))+O\left(\frac{T}{\log T}\right),
    \ \text{ as }\ T\to\infty.
  \end{align*}
\end{corollary}

Generally speaking, the philosophy behind the Phillips-Sarnak conjecture
\cite{PS85,Sa03}
suggests that the spectral analysis of the Laplacian acting on smooth
functions on a finite volume, hyperbolic Riemann surface $M$ should depend on
the arithmetic nature of the underlying Fuchsian group $\Gamma$.
The first terms in the asymptotic expansion in Corollary
\ref{ClassicalWeylLaw} depend solely on the volume of $X_{N}$, and then one
sees that the coefficient of $T$ depends on $N$.
For example, the groups corresponding to $N=5$ and $N=6$ have the same
signature, hence $X_{5}$ and $X_{6}$ have the same volume yet, by Corollary
\ref{ClassicalWeylLaw}, $X_{5}$ has infinitely more eigenvalues than $X_{6}$
in the sense that
\begin{align*}
  \lim\limits_{T \to \infty}
  \frac{\pi}{T}\left( \mathcal{N}_5(T) - \mathcal{N}_6(T)\right)
  = \log(6/5) > 0.
\end{align*}
Later in this article, we provide a list of further examples of topologically
equivalent surfaces associated to moonshine groups which have different Weyl's
laws.
We view these results as being consistent with and in support of the
Phillips-Sarnak philosophy.

Having established that the classical Weyl's law associated to $\Gamma_0(5)^+$
and $\Gamma_0(6)^+$ differ, we find it interesting to investigate other
conjectures concerning the distribution of eigenvalues.
Using the methodology from \cite{The12}, and references therein, we have
numerically computed sets of Maass cusp forms associated to $\Gamma_0(5)^+$
and $\Gamma_0(6)^+$.
On $\Gamma_0(5)^+$ our numerical results cover the range
$0<\lambda\le125^2+1/4$ which includes $3557$ Maass cusp forms, and on
$\Gamma_0(6)^+$ we cover the range $0<\lambda\le230^2+1/4$ which includes
$12474$ Maass cusp forms.
The distribution of the numerically found eigenvalues is in agreement with the
following conjecture.
\begin{conjecture}[Arithmetic Quantum Chaos \cite{BGGS92,BSS92}]\label{aqc}
  On surfaces of constant negative curvature that are generated by arithmetic
  fundamental groups, the distribution of the discrete eigenvalues of the
  hyperbolic Laplacian approaches a Poisson distribution as
  $\lambda\to\infty$.
\end{conjecture}

A particular feature of a Poisson distribution is the ``absence of memory'',
which, in our case, asserts that an eigenvalue cannot be predicted from
knowledge of all the previous eigenvalues.
The computation of eigenvalues allows us to verify that, numerically,
eigenvalues of the Laplacian on $X_{5}$ and $X_{6}$ are uncorrelated.

This paper is organized as follows.
In section \ref{prelim} we provide preliminary material for both the
theoretical and numerical aspects of our work.
Theoretically, we prove that the Riemann surfaces associated to the moonshine
groups $\Gamma_0(N)^+$ for square-free $N$ have one cusp, and we compute the
first Fourier coefficient of the corresponding non-holomorphic Eisenstein
series.
In order to make this article as self-contained as possible, we include a
discussion of Hejhal's algorithm for numerically estimating eigenvalues
together with Turing's method which is used to verify that no eigenvalue has
been missed.
In section \ref{secProof} we prove Theorem \ref{AverageWeylLaw}, and as
corollaries state the result in the cases of $\SL(2,\Z)$, $\Gamma_0(5)^+$ and
$\Gamma_0(6)^+$.
In section \ref{numerics} we state the conclusions from our numerical
investigations, and in section \ref{conclusions} we present various concluding
remarks.

\section{Preliminaries}\label{prelim}

\subsection{Moonshine groups $\Gamma_0(N)^+$}

In this subsection we will derive some important properties of moonshine
groups $\Gamma_0(N)^+$, for a square-free integer $N$.
We will prove they have exactly one cusp.
We then compute the constant Fourier coefficient of the associated
non-holomorphic Eisenstein series.
Equivalently, we compute the scattering determinant associated to the cusp.
We refer to \cite{Hej83} and \cite{Iw02} for relevant background information.

\begin{lemma}
  For every square-free integer $N>1$, the surface $X_N$ has exactly one cusp,
  which can be taken to be at $i\infty$.
\end{lemma}
\begin{proof}
  The cusps of $X_N$ are uniquely determined by parabolic elements of the
  group $\Gamma_0(N)^+$.
  In \cite{Cum04} it is proved that all parabolic elements of $\Gamma_0(N)^+$
  have integral entries.
  Therefore, the parabolic elements of $\Gamma_0(N)^+$ are also parabolic
  elements of the congruence group $\Gamma_0(N)$.
  From pages 44--47 of \cite{Iw02}, we easily deduce that the only possible
  cusps of $\overline{\Gamma_0(N)^+}\backslash\h$ belong to the set
  $\{0,i\infty\} \cup \{1/v : v \mid N\}$.
  The point $z=0$ is mapped to $i\infty$ by involution
  \begin{align*}
    \begin{pmatrix}0&-1/\sqrt{N}\\\sqrt{N}&0\end{pmatrix}\in\Gamma_0(N)^+.
  \end{align*}

  For an arbitrary $v \mid N$ and $w=N/v$ one has $(w,v)=1$ since $N$ is
  square-free.
  By Euclid's algorithm, there exists integers $a$ and $b$ such that
  $-aw-bv=1$.
  Therefore, points $z=1/v$ are mapped to $i\infty$ by transformation
  \begin{align*}
    \frac{1}{\sqrt{w}}\begin{pmatrix}aw&b\\N&-w\end{pmatrix}\in\Gamma_0(N)^+.
  \end{align*}
  This shows that all possible cusps of $\overline{\Gamma_0(N)^+}\backslash\h$
  are $\Gamma_0(N)^+$-equivalent with $i\infty$.
  Therefore $\overline{\Gamma_0(N)^+}\backslash\h$ has exactly one cusp which
  can be taken to be $i \infty$, as claimed.
\end{proof}

Let $\zeta(s)$ denote the (classical) Riemann zeta function and let $\xi(s)$
be the completed zeta function, defined by
$\xi(s):=\frac{1}{2}s(s-1)\pi^{-s/2}\Gamma(s/2)\zeta(s)$.

\begin{lemma}
  For a square-free, positive integer $N=p_1\cdots p_r$, the scattering
  determinant associated to the cusp of $X_N$ at $i\infty$ is given by the
  following expression
  \begin{align}\label{DefScattViaXi}
    \varphi_N(s)=\frac{s}{s-1}\frac{\xi(2s-1)}{\xi(2s)}\cdot D_N(s),
  \end{align}
  where
  \begin{align*}
    D_N(s):=\frac{1}{N^s}\cdot\prod_{j=1}^r\frac{p_j^s+p_j}{p_j^s+1}.
  \end{align*}
\end{lemma}
\begin{proof}
  By Theorem 3.4 from \cite{Iw02} we write
  $\varphi_N(s)= \sqrt{\pi} \Gamma(s-1/2) \Gamma^{-1}(s) H_N(s)$, where
  $H_N (s)$ denotes the Dirichlet series portion of the scattering determinant.
  Let $C_N$ denote the set of left-lower entries of matrices from
  $\Gamma_0(N)^+$.
  Following pages 45--49 from \cite{Iw02}, one sees that
  \begin{align*}
    H_N(s)=\sum_{c\in C_N} c^{-2s} \A_N(c)
  \end{align*}
  is well defined for $\Re(s)>1$, where $\A_N(c)$ is equal to the
  number of distinct values of $d$ modulo $c$ such that $c$ and $d$ are
  elements of the bottom row of the matrix from $\Gamma_0(N)^+$.

  From the definition of $\Gamma_0(N)^+$, we easily deduce that
  $C_N=\{(N/\sqrt{v})\cdot n: v\mid N, n\in \N\}$.

  For a fixed $c=(N/\sqrt{v})\cdot n$, with $v\mid N$ and $n\in\N$
  arbitrary, we can take $e=v$ in the definition of $\Gamma_0(N)^+$ to
  deduce that matrices from $\Gamma_0(N)^+$ with left lower entry $c$ are
  given by
  \begin{align*}
    \begin{pmatrix}\sqrt{v}a&b/\sqrt{v}\\\frac{N}{v}\sqrt{v}n&\sqrt{v}d
    \end{pmatrix}
  \end{align*}
  for some integers $a$, $b$ and $d$ such that $vad-(N/v)bn=1$.
  Therefore, the number $\A_N((N/\sqrt{v})\cdot n)$ is equal to the
  number of distinct solutions $d$ modulo $(N/v)n$ of the equation
  $vad-(N/v)bn=1$.
  Since $N$ is square-free, this equation has a solution if and only if
  $(v,n)=1$ and $(d,(N/v)n)=1$.
  In this case, the number of distinct solutions $d$ modulo $(N/v)n$ is equal
  to $\varphi((N/v)n)$.
  Here, $\varphi$ denotes the Euler totient function and $(p,q)$ denotes the
  greatest common divisor of integers $p$ and $q$.

  Therefore, $\A_N((N/\sqrt{v})\cdot n)=0$ if $(v,n) \neq 1$ and
  $\A_N((N/\sqrt{v})\cdot n)=\varphi((N/v)n)$ if $(v,n)=1$.
  Now, we may conclude that
  \begin{align*}
    H_N(s)=\sum_{v \mid N}\sum_{(n,v)=1}\frac{\varphi\left(\frac{N}{v}n
      \right)}{\left(\frac{N}{v}\sqrt{v} n\right)^{2s}}.
  \end{align*}

  The inner sum on the right-hand side of the above equation may be expressed
  using computations from \cite{Hej83}, specifically Lemmata 4.5 and 4.6 on
  page 535, showing that for positive integers $A_1$, $A_2$, $B_1$ and $B_2$
  one has
  \begin{multline}\label{HejhalSum}
    \sum_{c_0>0\text{: } (c_0, (B_1,A_2)(A_1,B_2))=1}
    \frac{\varphi(c_0\cdot(B_1, B_2)(A_1, A_2))}{c_0^{2s} (A_1, A_2)^s
      (B_1,B_2)^{2s}} \\
    =\frac{\zeta(2s-1)}{\zeta(2s)}\cdot\prod_{p \mid (A_1, A_2)(B_1,B_2)}
    \left(\frac{p-1}{p^{2s-1}}\right) \prod_{p \mid (A_2, B_1)(A_1,B_2)}
    \left(\frac{p^s-p^{1-s}}{p^{2s-1}}\right);
  \end{multline}
  in standard notation, $p$ denotes a prime number, and an empty product is
  defined to be equal to $1$.

  Using formula \eqref{HejhalSum} with $A_1=v$, $A_2=1$; $B_1=N/v$,
  $B_2=N$ and the principle of mathematical induction with respect to the
  number $r$ of distinct prime factors of $N=p_1\cdots p_r$, we deduce that
  \begin{align*}
    H_N(s)=\frac{\zeta(2s-1)}{\zeta(2s)}\cdot\sum_{v \mid N}\left(
    \prod_{p \mid\left(\frac{N}{v}\right)}\frac{p-1}{p^{2s}-1} \prod_{p \mid v}
    \frac{p^s-p^{1-s}}{p^{2s}-1}\right)
    =\frac{\zeta(2s-1)}{\zeta(2s)}\cdot\frac{1}{N^s}\cdot
    \prod_{j=1}^r\frac{p_j^s+p_j}{p_j^s+1}.
  \end{align*}
  Therefore, the scattering matrix $\varphi_N(s)$, for $\Re (s)>1$ is given by
  \begin{align*}
    \varphi_N(s)=\sqrt{\pi}\frac{\Gamma(s-1/2)}{\Gamma(s)}\cdot
    \frac{\zeta(2s-1)}{\zeta(2s)}\cdot D_N(s),
  \end{align*}
  The statement of the lemma follows from the definition of the completed zeta
  function, which completes the proof of the Lemma.
\end{proof}

\begin{remark}\rm
  The determinant of the scattering matrix for congruence subgroups has
  been computed by Hejhal \cite{Hej83} and Huxley \cite{Hu84}.
\end{remark}

\subsection{Moonshine groups $\Gamma_0(5)^+$ and $\Gamma_0(6)^+$}

The moonshine group $\Gamma_0(5)^+$ is generated by
\begin{align*}
  g_1=\begin{pmatrix}1&1\\0&1\end{pmatrix}, \quad
  g_2=\frac{1}{\sqrt{5}} \begin{pmatrix}5&-1\\5&0\end{pmatrix}, \quad
    g_3=\frac{1}{\sqrt{5}}\begin{pmatrix}5&-3\\10&-5\end{pmatrix},
\end{align*}
and the moonshine group $\Gamma_0(6)^+$ is generated by
\begin{align*}
  g_1=\begin{pmatrix}1&1\\0&1\end{pmatrix}, \quad
  g_2=\frac{1}{\sqrt{6}}\begin{pmatrix}6&-1\\6&0\end{pmatrix}, \quad
    g_3=\frac{1}{\sqrt{3}}\begin{pmatrix}3&-2\\6&-3\end{pmatrix},
\end{align*}
see \cite{Cum10}.
Fundamental domains of $X_5=\overline{\Gamma_0(5)^+}\backslash\h$ and
$X_6=\overline{\Gamma_0(6)^+}\backslash\h$ are displayed in figure
\ref{fund dom}.
For both, $X_5$ and $X_6$, the sides are identified according to the pairings
\begin{align*}
  g_1:s_1\mapsto s_6, \quad
  g_2:s_2\mapsto s_5, \quad
  g_3:s_3\mapsto s_4.
\end{align*}
Both $X_5$ and $X_6$ have a cusp at $v_1=i\infty$, and each surface has three
inequivalent elliptic fixed points which are all of order $2$.
The elliptic fixed points are
\begin{align*}
  g_1^{-1}g_2&:v_2\mapsto v_2 \quad\text{which is $\Gamma$-equivalent with}
  \quad v_6=g_1 v_2, \\
  g_2^{-1}g_3&:v_3\mapsto v_3 \quad\text{which is $\Gamma$-equivalent with}
  \quad v_5=g_3 v_3, \\
  \text{and} \qquad g_3&:v_4\mapsto v_4.
\end{align*}
By the Gauss-Bonnet theorem, the volumes of the surfaces are $\vol(X_5)=\pi$
and $\vol(X_6)=\pi$.

\begin{figure}
  \includegraphics{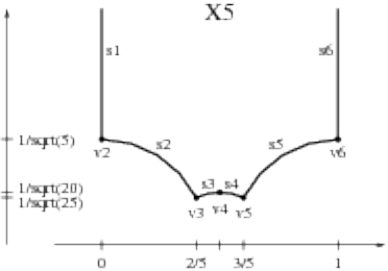} \hfill \includegraphics{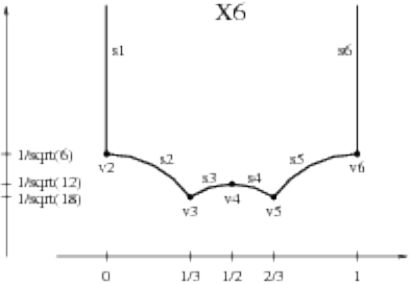}
  \caption{\label{fund dom}Dirichlet fundamental domains of the moonshine
    groups $\Gamma_0(5)^+$ (left), and $\Gamma_0(6)^+$ (right).}
\end{figure}

\subsection{Str\"{o}mbergsson's pullback algorithm}

In 2000, Str\"{o}mbergsson \cite{St00} presented an algorithm for
computing the pullback of any point $z \in \h$ into the Dirichlet fundamental
domain of a given cofinite Fuchsian group $\Gamma$ with prescribed generators.
Str\"{o}mbergsson's algorithm uses only the action of generators of the group
$\Gamma$ applied to the point $z$,
and the algorithm is shown to converge after a finite number of iterations.
Computation of the pullback of a point $z \in \h$ to the Dirichlet fundamental
domain of $\Gamma$ is an ingredient in Hejhal's algorithm for computing
Maass forms, recalled below.
Therefore, Str\"{o}mbergsson's algorithm is an important part of our numerical
computations of eigenvalues of Maass forms on $X_5$ and $X_6$.

For the sake of completeness, we will recall the Str\"{o}mbergsson algorithm
in its full generality.
Assume that $\Gamma$ is a cofinite Fuchsian group with generators
$g_1, \ldots, g_n$ and set of elliptic fixed points $\mathcal{E}$.
Let $d(z,w)$ denote the hyperbolic distance between two points $z$ and $w$ in
$\h$.
The associated Dirichlet fundamental domain is the set
\begin{align*}
  \F=\{z\in\h\ |\ d(p,z)\le d(p,\gamma z)\ \forall\gamma\in\Gamma\},
\end{align*}
where $p\in \h - \mathcal{E}$ is arbitrary.
The given generators of $\Gamma$ identify the sides of $\F$.
Str\"{o}mbergsson's algorithm for computing the pullback
of any point $z \in \h$ into $\F$ is the following.

\begin{alg}[Pullback algorithm \cite{St00}]
  Choose any $z\in\h$.
  \begin{enumerate}
  \item\label{alg:pullback.2} Compute the $2n$ points
    $g_1z,g_1^{-1}z,g_2z,g_2^{-1}z,\ldots,g_n^{-1}z$.
    Let $z'$ be the one of these points which has the {\em smallest}
    hyperbolic distance to $p$.
  \item If $d(p,z')<d(p,z)$, then {\em replace $z$ by $z'$}, and
    repeat with step \ref{alg:pullback.2}.
  \item If $d(p,z')\ge d(p,z)$, then we {\em know that $z$
    lies in $\F$}, hence $z$ is the desired point, i.e.\ the pullback of the
    point initially selected.
  \end{enumerate}
\end{alg}
Str\"ombergsson proved that his algorithm always finds the pullback
within a finite number of operations \cite{St00}.

We use $z^*=x^*+iy^*$ to denote the pullback of $z=x+iy$.

\subsection{Maass forms on $\Gamma_0(N)^+$}

Let us recall the definition of Maass forms \cite{Maa49} and Maass cusp forms.
\begin{definition}
  $f:\h\to\R$ is a Maass form on $\Gamma_0(N)^+$ associated to the eigenvalue
  $\lambda$ if and only if
  \begin{enumerate}
  \item[i)] $f\in C^\infty(\h)$,
  \item[ii)] $f\in L^2(X_N)$,
  \item[iii)] $-\Delta f(z)=\lambda f(z)$,
  \item[iv)] $f(\gamma z)=f(z)\ \forall\gamma\in\Gamma_0(N)^+$.
  \end{enumerate}
\end{definition}
\begin{definition}
  $f:\h\to\R$ is a Maass cusp form on $\Gamma_0(N)^+$ if and only if
  \begin{enumerate}
  \item[i)] $f$ is a Maass form on $\Gamma_0(N)^+$,
  \item[ii)] $\lim_{z\to i\infty}f(z)=0$.
  \end{enumerate}
\end{definition}
For $z=x+iy \in \h$, the Fourier expansion of a Maass cusp form associated to
the eigenvalue $\lambda=r^2+1/4$ is given by
\begin{align}\label{FourierExpMaass}
  f(x+iy)=\sum_{n\in\Z-\{0\}}a_ny^{1/2}K_{ir}(2\pi|n|y)e^{2\pi inx},
\end{align}
where $K$ stands for the $K$-Bessel function.
Since a Maass form is real analytic, we have
$\Re a_{-n}=\Re a_n$ and $\Im a_{-n}=-\Im a_n$.

As first proved in \cite{Maa49}, the spectral coefficients $a_n$ grow at most
polynomially in $n$.
The $K$-Bessel function decays exponentially for large arguments, meaning
\begin{align*}
  K_{ir}(y)\sim\sqrt{\frac{\pi}{2y}}e^{-y}\ \text{ for }\ y\to\infty.
\end{align*}
As a result, one can obtain a very good approximation of the expansion
\eqref{FourierExpMaass} by using finitely many terms, where the number of
terms considered depends on the desired accuracy of the approximation.

Let $\F_N\simeq\overline{\Gamma_0(N)^+}\backslash\h$ be the fundamental domain
of $\Gamma_{0}(N)^{+}$.
Let $z^*=x^*+iy^*$ be the $\Gamma _0(N)^+$-pullback of the point $z=x+iy$ into
the fundamental domain, meaning there exists some $\gamma\in\Gamma_0(N)^+$
such that $z^*=\gamma z$ and $z^*\in\F_N$.
By the definition of automorphy, we have that $f(z)=f(z^*)$.

Since the congruence group $\Gamma _0(N)$ is a subgroup of $\Gamma _0(N)^+$,
we immediately deduce the following lemma.
\begin{lemma}\label{MaassFormsLemma}
  If $f$ is a Maass form on $\Gamma _0(N)^+$, then $f$ is a Maass form on
  $\Gamma _0(N)$.
\end{lemma}

\subsection{Hecke operators}

Let us recall the definition of Hecke operators.
There are many references for
this material, one of which being \cite{Sh71}.

\begin{definition}
  Let $f:\h\to\R$, and $n$ a positive integer.
  The Hecke operator $T_n$ is defined by
  \begin{align*}
    T_nf(z)=\frac1{\sqrt{n}}\sum_{\substack{ad=n\\d>0}}\sum_{b=0}^{d-1}
    f(\frac{az+b}d).
  \end{align*}
\end{definition}

\begin{theorem}[\cite{AL70,Sh71}]\label{HeckeOpTheorem}
  Consider the congruence group $\Gamma_0(N)$.
  For all $n$ such that $(n,N)=1$, the Hecke operators $T_n$ are
  endomorphisms of the space of Maass cusp forms on $\Gamma_0(N)$.
  For all $m$ and $n$ with $(m,N)=(n,N)=1$ and all Maass cusp
  forms $f(z)$ on $\Gamma_0(N)$, the Hecke operators have the following
  properties:
  \begin{align*}
    &T_mT_n=\sum_{d|(m,n)}T_{\frac{mn}{d^2}}, \\
    &T_n\circ\Delta=\Delta\circ T_n, \\
    &T_nf(z)=t_nf(z),
  \end{align*}
  where the eigenvalues $t_n$ of the Hecke operators $T_n$ are related to the
  expansion coefficients $a_n$ of the Maass cusp form $f(z)$ by the identity
  \begin{align*}
    a_n=a_1t_n.
  \end{align*}
\end{theorem}
For a proof of the theorem, see \cite{AL70}, \cite{Sh71}, or \cite{St12}.

Theorem \ref{HeckeOpTheorem} immediately implies that the Fourier
coefficients of Maass cusp forms on $\Gamma_0(N)$ are multiplicative,
\begin{align*}
  a_ma_n=a_1\sum_{\substack{d|(m,n)\\d>0}}a_{\frac{mn}{d^2}}
\end{align*}
for all $m$ and $n$ with $(m,N)=(n,N)=1$.
By Lemma \ref{MaassFormsLemma}, this holds also for Maass cusp forms on
$\Gamma_0(N)^+$.

\subsection{Hejhal's algorithm}\label{HejhalAlgorithm}

We make use of Hejhal's algorithm \cite{Hej99,The05} which itself employs the
Fourier expansion \eqref{FourierExpMaass} of Maass cusp forms.

Hejhal's algorithm is a finite system of linear equations whose non-trivial
solutions are related to Maass cusp forms.
Hejhal's algorithm is heuristic.
By construction, a Maass cusp form always will solve the linear
equations of the algorithm to any desired level of accuracy, but the
converse is not true.
Not each solution of the finite system of linear equations is a Maass cusp
form.
Only in the case when a solution is independent of the parameters will
the solution approximate a Maass cusp form.
The crucial parameter in question is the choice of the value of $y$ in
\eqref{Hejhal}.
The computation of Maass cusp forms therefore proceeds in two steps:
Heuristic use of Hejhal's algorithm, followed by a verification of the
numerical results.

Theoretically, Maass cusp forms can be rigorously certified as was shown in
\cite{BSV06} in the example of the modular group.
Using the quasi-mode construction, Booker, Str\"ombergsson, and Venkatesh
have certified the first $10$ eigenvalues of $\SL(2,\Z)$.
The certification techniques can be adopted to other settings,
such as to the moonshine groups.
Practically, however, we have to bear in mind that rigorously certifying
eigenvalues requires immense computer resources and it is infeasible to
certify thousands of Maass cusp forms.
For this reason, we just verify the numerical results with a different,
{\em not fully rigorous} method.

The verification is based on the following:
\begin{enumerate}
\item Fix $y$.
\item Find non-trivial solutions of Hejhal's system of linear equations.
\item Take a finite number of different values of $y$, and check
  whether the non-trivial solutions seem to be independent of $y$.
\item Take only the solutions which are seemingly independent of $y$ and
  make a list of conjectured Maass cusp forms.
\end{enumerate}
In the end, there will be strong evidence, but not a proof, that the list of
conjectured Maass cusp forms is indeed a list of true Maass cusp forms.
It is the experience of those who implement the algorithm that
{\em more than half} of the non-trivial solutions of Hejhal's system of
equations for a fixed value of $y$ are {\em not} Maass cusp forms.
Taking a second choice for $y$ immediately rules out {\em almost all}
solutions which are not a Maass cusp form.

There remains the possibility that a solution could solve Hejhal's linear
system of equations for two independent values of $y$ whilest not being
a Maass cusp form.
We have further checked whether this has happened by employing several
independent values of $y$.
Empirically, it turned out that as soon as some function solves Hejhal's
system of equations for two independent values of $y$, it does so for
any finite number of independent values of $y$ also.
And we conjecture that it does so for any other value of $y$.

Further evidence comes from a second verification based on the Hecke operators.
According to the Hecke operators, the expansion coefficients of Maass forms
are multiplicative.
When solving Hejhal's system of linear equations, there is no reason that
the coefficients of a solution are multiplicative, but only
those solutions whose coefficients are multiplicative can be Maass cusp forms.

Numerically, for {\em each individual} solution of Hejhal's system of linear
equations we have investigated and found that a solution is seemingly
independent of $y$ {\em if and only if} the expansion coefficients of the
solution are multiplicative.
This means both verifications agree in their answer.

Let us now recall Hejhal's algorithm.

Since $\Gamma_0(N)^+$ is cofinite and has only one cusp at $i\infty$, we can
bound $y$ from below.
Allowing for a small numerical error of at most $[[\eps]]$, where $[[\eps]]$
stands for $|\textit{numerical error}|\lesssim\eps$, due to the
exponential decay of the $K-$Bessel function in $y$, we can truncate the
absolutely convergent Fourier expansion \eqref{FourierExpMaass} such that
\begin{align}\label{TruncatedExp}
  f(x+iy)=\sum_{0\not=|n|\le M(\eps,r,y)}a_ny^{1/2}K_{ir}(2\pi|n|y)e^{2\pi inx}
  +[[\eps]].
\end{align}
Solving for the spectral coefficients results in the equation
\begin{align}\label{a_m}
  a_my^{1/2}K_{ir}(2\pi|m|y)=\frac{1}{2Q}\sum_{j=1}^{2Q}f(\frac{j}{2Q}+iy)
  e^{-2\pi im\frac{j}{2Q}}+[[\eps]],
\end{align}
with $2Q>M+m$.

By automorphy, any Maass cusp form can be approximated by
\begin{align}\label{f(z^*)}
  f(x+iy)=f(x^*+iy^*)=\sum_{0\not=|n|\le M_0}a_n{y^*}^{1/2}K_{ir}(2\pi|n|y^*)
  e^{2\pi inx^*}+[[\eps]],
\end{align}
where $y^*$ is always larger than or equal to the height of the lowest point
of the fundamental domain $\F$, allowing us to replace
$M(\eps,r,y^*)$ by $M_0=M(\eps,r,\min_{w\in\F}\Im w)$.

Making use of the implicit automorphy by replacing $f(x+iy)$ in \eqref{a_m}
with the right-hand side of \eqref{f(z^*)} yields
\begin{align}\label{Hejhal}
  a_my^{1/2}K_{ir}(2\pi|m|y)=\frac{1}{2Q}\sum_{j=1}^{2Q}\sum_{0\not=|n|\le M_0}
  a_n{y_j^*}^{1/2}K_{ir}(2\pi|n|y_j^*)e^{2\pi i(nx_j^*-mx_j)}+[[2\eps]],
  \ \text{ where }\ x_j+iy_j=\frac{j}{2Q}+iy,
\end{align}
for $0\not=|m|\le M(\eps,r,y)$, which is the central identity of the algorithm.

We are looking for non-trivial solutions numerically such that \eqref{Hejhal}
vanishes simultaneously for all $0\not=|m|\le M_0$ and $0<y<\min_{w\in\F}\Im w$.
Each non-trivial solution gives a Maass cusp form whose eigenvalue reads
$\lambda=r^2+1/4$.

We first solve \eqref{Hejhal} for all $0\not=|m|\le M_0$ numerically,
but use a single value of $y$ only.
Then, we verify with a finite number of values of $y$, whether we have found a
non-trivial solution such that \eqref{Hejhal} vanishes simultaneously for all
$0\not=|m|\le M_0$ for each value of $y$.
If the solution turns out to be seemingly independent of $y$, we finally check
whether the expansion coefficients $a_n$ are multiplicative.
If also the expansion coefficients turn out to be multiplicative, we have
verified that the numerically found solution of \eqref{Hejhal} is a Maass cusp
form.

Let us now specify good parameter values for solving \eqref{Hejhal}
numerically.
\begin{algorithm}[Parameter values]
  Let $\tilde\lambda=t^2+1/4$ be close to an eigenvalue.
  Let the precision be given by $\eps>0$.
  Then for $\lambda$ near $\tilde\lambda$ we choose the values of the
  parameters as follows:
  \begin{enumerate}
  \item Solve $\eps K_{it}(\max\{t,1\})=K_{it}(2\pi M_0\min_{w\in\F}\Im w)$
    in $M_0$ with $2\pi M_0\min_{w\in\F}\Im w>\max\{t,1\}$.
  \item Let
    $\displaystyle y=\frac{9}{10}\frac{\max\{t,1\}}{2\pi M_0}$.
  \item\label{parameters.M} Solve
    $\eps K_{it}(\max\{t,1\})=K_{it}(2\pi My)$ in $M$
    with $2\pi My > \max\{t,1\}$, i.e.\
    $\displaystyle M=\frac{\min_{w\in\F}\Im w}{y}M_0$.
  \item Let $Q$ be the smallest integer which is larger than $M$.
  \item Check whether \eqref{Hejhal} is well conditioned for the given $y$
    and all $0\not=|m|\le M_0$.
    If not, reduce $y$ slightly and repeat with \ref{parameters.M}.
  \end{enumerate}
\end{algorithm}
For verifying that \eqref{Hejhal} vanishes simultaneously for all
$0\not=|m|\le M_0$ for a finite number of values of $y$, we use
$\displaystyle y=\frac{\max\{t,1\}}{2\pi M_0}$, and check whether
\eqref{Hejhal} is well conditioned.
If \eqref{Hejhal} is not well conditioned, we reduce $y$ slightly.
The algorithm ensures that we never reduce $y$ by a factor of $9/10$ or more.
Now we check whether \eqref{Hejhal} vanishes simultaneously for all
$0\not=|m|\le M_0$ for the given $y$.
If \eqref{Hejhal} does vanish, we continue with a finite number of random
choices for the value of
$\displaystyle
y\in\Big(\frac{9}{10}\frac{\max\{t,1\}}{2\pi M_0},\min_{w\in\F}\Im w\Big]$
and check for each value of $y$ whether \eqref{Hejhal} vanishes for all
$0\not=|m|\le M_0$.

\subsection{Turing's method}\label{Turing}

Turing's method is a method of verification that the list of eigenvalues of
Maass cusp forms is \textit{consecutive}, once we have a suitable bound for
the error term $S(t)$ in ``average'' Weyl's law $\mathcal{N}(t)=\M(t)+S(t)$.
Roughly speaking, the method is the following.
Assume that the error term $S(t)$ in the ``average'' Weyl's law for the
corresponding surface satisfies a bound of the type
\begin{align*}
  E_l(T) \le \left\langle S(T)\right\rangle:= \frac{1}{T} \int \limits_{0}^{T}
  S(t)dt \le E_u(T),
\end{align*}
where $E_l(T) \to 0$ and $E_u(T) \to 0$, as $T \to \infty$.
Then, we have the following test of consecutiveness \cite{Boo06,Tur53}.

Step 1. Compute $\mathcal{N}^{\text{num}}(T)$; the number of numerically found
eigenvalues in the interval $1/4<\lambda\le T^2+1/4$ and denote by
\begin{align*}
  S^{\text{num}}(T):=\mathcal{N}^{\text{num}}(T)-\M(T)
\end{align*}
the difference between the number of numerically found eigenvalues and the
average Weyl's law.

Step 2. Add a ``fake'' eigenvalue $\lambda_{\textrm{fake}}$ near the end of the
list of eigenvalues and compute $\left\langle S^{\text{num}}(T) \right\rangle$.
If the value $\left\langle S^{\text{num}}(T) \right\rangle$ exceeds $E_u(T)$,
then the list of eigenvalues is consecutive in the interval
$1/4 < \lambda \le \lambda_{\textrm{fake}}$.

\section{Average Weyl's law for $\Gamma_0(N)^+$}\label{secProof}

In this section we prove Theorem \ref{AverageWeylLaw}.

Let us recall that $N=p_1\cdots p_r$ is a square-free positive integer, and
define the function
\begin{align*}
  \alpha_N(j,T):=T\log p_j-\lfloor\frac{T\log p_j}{\pi} \rfloor\pi,
\end{align*}
where, as previously stated, $\lfloor x\rfloor$ denotes the greatest integer
less than or equal to $x$.
Let $X_N=\overline{\Gamma_0(N)^+}\backslash\h$ be the Riemann surface
associated to the Fuchsian group $\Gamma_{0}^{+}(N)$, and let $Z_{X_N}$ denote
the Selberg zeta function associated to $X_N$.

Let $A\in (1,3/2)$ and $T>1$ be arbitrary real numbers, and let
$R(A)$ be the rectangle with vertices $1-A-iT$, $A-iT$, $A+iT$, $1-A+iT$.
Without loss of generality, we assume that $A$ and $T$ are such that
$Z_{X_N}(s) \neq 0$ for $s\in\partial R(A)$.
Formula (5.3) on p.~498 of \cite{Hej83} states the location of zeros
and poles of the Selberg zeta function $Z_{X_N}$.
In the notation of \cite{Hej83}, one has that $m=0$ and $W=\operatorname{id}$.
Furthermore, $\varphi_N(1/2)=-1$, hence, application of Theorem 4.1 on p.~482
and formula (4.6) on p.~485 of \cite{Hej83} yields that, in the notation of
formula (5.3), one has $A+B-K_0-C=2g-2$.
Therefore,
\begin{align*}
  \frac{1}{2\pi i} \int_{\partial R(A)}\frac{Z_{X_N}'}{Z_{X_N}}(s)ds
  =2\mathcal{N}_N[0<r_n\le T]+2Q_N[0<\Im (\rho)\le T]+2g-2+n_N+2m_{1/4, N},
\end{align*}
where $Q_N[0<\Im (\rho)\le T]$ denotes the number of zeros $\rho$
of the scattering determinant $\varphi_N$ with $\Im (\rho)\in (0, T]$.

Let $\partial P(A)$ denote the polygonal path joining points
$1/2-iT$, $A-iT$, $A+iT$ and $1/2+iT$.
Using the functional equation for the function
$\mathcal{D}_N(s):=\frac{Z_{X_N}'}{Z_{X_N}}(s)$, as in the proof of
Theorem 2.28 on pp.~466--467 of \cite{Hej83}, we can write
\begin{align}\label{sumzeros}
  \mathcal{N}_N(T)+Q_N[0<\Im (\rho)\le T]=1-g-\frac{n_N}{2}-m_{1/4, N}
  +R_1(T)+\frac{1}{4\pi i}\int\limits_{\partial P(A)}
  \frac{\varphi_N'}{ \varphi_N}(s)ds-\frac{1}{4\pi i}
  \int\limits_{\partial P(A)} \mathcal{C}_N(s)ds,
\end{align}
where
\begin{align*}
  R_1(T) :=\frac{1}{2\pi i}\int\limits_{\partial P(A)} \mathcal{D}_N(s) ds
\end{align*}
and
\begin{multline}\label{Function C}
  \mathcal{C}_N(s)=\vol(X_N)(s-1/2)\tan(\pi(s-1/2))
  -\sum_{i=1}^{l}\sum_{j=1}^{m_i-1}\frac{\pi}{m_i\sin (\pi j/m_i)}
  \frac{\cos\pi(2j/m_i-1)(s-1/2)}{\cos\pi(s-1/2)} \\
  +2\log 2+\frac{\Gamma^{\prime }}{\Gamma}(1/2+s)
  +\frac{\Gamma^{\prime }}{\Gamma}(3/2-s).
\end{multline}
By Theorem 2.29 on p.~468 of \cite{Hej83}, we have the estimates
\begin{align}\label{R_1 bound}
  R_1(T)=O\left(\frac{T}{\log T}\right)\ \text{ and }
  \ \int\limits_0^T R_1(t)dt=O\left(\frac{T}{\log^2T}\right)
  \ \text{ as }\ T\to\infty.
\end{align}
To see that our function $R_{1}(T)$ is equal to the function $S(T)$ in
Theorem 2.29 of \cite{Hej83}, we refer to Definition 2.27 on page 465 of
\cite{Hej83}.
In addition, one can easily prove that $S_{1}(T)$, in the notation of
\cite{Hej83}, coincides with the integral of $R_{1}(T)$.
To do so, one simply integrates the formula for $R_{1}(T)$, interchanges the
order of integration, evaluates the inside integral, and then integrates by
parts.
We choose to omit the details of these elementary calculations.

Since $A\in (1,3/2)$, the function $\mathcal{C}_N(s)$ has no poles on the
sides of the rectangle $R_{1/2}(A)$ which has vertices at the points
$1/2-iT$, $A-iT$, $A+iT$ and $1/2+iT$.
Furthermore, the only pole of $\mathcal{C}_N(s)$ inside $R_{1/2}(A)$ is a
simple pole at $s=1$.
Since
\begin{align*}
  \lim_{s\to 1}\frac{s-1}{\cos\pi(s-1/2)}=-\frac{1}{\pi}
\end{align*}
we conclude that
\begin{align*}
  \res_{s=1} \mathcal{C}_N(s)=-\frac{\vol(X_N)}{2\pi}
  +\sum_{i=1}^{l}\sum_{j=1}^{m_i-1}
  \frac{\cos\pi(j/m_i-1/2)}{m_i\sin(\pi j/m_i)}
  =-\frac{\vol(X_N)}{2\pi}+\sum_{i=1}^{l}\left( 1-\frac{1}{m_i}\right)=1-2g.
\end{align*}
Therefore, by the calculus of residues, having in mind that
$\mathcal{C}_N(1/2+it)=\mathcal{C}_N(1/2-it)$, for real and non-negative $t$ we
get
\begin{align}\label{Integral of C_1}
  \frac{1}{4\pi i}\int\limits_{\partial P(A)} \mathcal{C}_N(s)ds
  =\frac{1}{2}(1-2g)+\frac{1}{2\pi}\int\limits_0^T\mathcal{C}_N(1/2+it)dt.
\end{align}
By substituting $s=1/2+it$ into \eqref{Function C}, we then have that
\begin{multline}\label{Integral of C_at 1/2}
  \int\limits_0^T \mathcal{C}_N (1/2+it)dt
  =-\vol(X_N)\int\limits_0^T t\tanh(\pi t)dt
  -\sum_{i=1}^{l}\frac{\pi}{m_i}\sum_{j=1}^{m_i-1}
  \frac{1}{\sin(\pi j/m_i)}\int\limits_0^T
  \frac{\cosh\pi(2j/m_i-1)t}{\cosh\pi t}dt \\
  +2 \Re\left(\int\limits_0^T
  \frac{\Gamma'}{\Gamma}(1+it)dt\right)+2 T\log 2
  =I_1 (T)-I_2(T)+I_3(T)+2T\log 2,
\end{multline}
where, in obvious notation, $I_{1}$, $I_{2}$ and $I_{3}$ are defined
to be the integrals in (\ref{Integral of C_at 1/2}).
We will now estimate each of these integrals.

We write $t\tanh(\pi t)=t-2t/(1+\exp(2\pi t))$ to get the expression
\begin{align*}
  I_1(T)=-\vol(X_N)\left(\frac{T^2}{2}-2\int\limits_0^{\infty}
  \frac{t dt}{1+e^{2\pi t}}+2g_1(T)\right),
\end{align*}
where
\begin{align}\label{g_1 def}
  g_1(T)=\int\limits_T^{\infty}\frac{t dt}{1+e^{2\pi t}}.
\end{align}
Quoting formula 3.411.3 from \cite{GR07} with $\nu=2$ and
$\mu=2\pi$, having in mind that $\zeta (2) = \pi^2 /6$ and $\Gamma(2)=1$,
we get
\begin{align}\label{I_1}
  I_1(T)=-\vol(X_N)\left(\frac{T^2}{2}-\frac{1}{24}+2 g_1(T)\right).
\end{align}

Similarly, by quoting formula 3.511.4 from \cite{GR07} with $a=\pi(2j/m_i-1)$
and $b=\pi$, we arrive at the equation
\begin{align*}
  \int\limits_0^T\frac{\cosh\pi(2j/m_i-1)t}{\cosh\pi t}dt
  =\frac{1}{2 \sin(\pi j/m_i)}-g_2(i,j,T)
\end{align*}
where
\begin{align*}
  g_2(i,j,T):=\int\limits_T^{\infty}
  \frac{\cosh\pi(2j/m_i-1)t}{\cosh\pi t}dt.
\end{align*}
Hence,
\begin{align}\label{I_2}
  I_2(T)=\sum_{i=1}^{l}\frac{\pi}{m_i}\sum_{j=1}^{m_i-1}
  \frac{1}{2\sin^2(\pi j/m_i)}-g_2(T),
\end{align}
where we define
\begin{align*}
  g_2(T):=\sum_{i=1}^{l}\frac{\pi}{m_i}\sum_{j=1}^{m_i-1}
  \frac{g_2(i,j,T)}{\sin(\pi j/m_i)}
\end{align*}

Finally, quoting formula 8.344 from \cite{GR07}, which is essentially
Stirling's formula, with $z=1+iT$ and $n=2$ we get that
\begin{align}\label{I_3}
  \frac{1}{2}I_3(T)=\Re\left(-i\int\limits_0^T(\log\Gamma(1+it))'dt
  \right)=\Im (\log\Gamma(1+iT))=-T+T\log T+\frac{\pi}{4}+g_3(T),
\end{align}
where
\begin{align}\label{g_3}
  g_3(T)=\frac{1}{2} \Im\left(\log\left( 1+\frac{1}{iT}\right)\right)
  +\Re\left( T\log\left( 1+\frac{1}{iT}\right)\right)
  -\frac{B_2 T}{2 (1+T^2)}+\Im ( \mathcal{R}_3(T))
\end{align}
and
\begin{align}\label{R_3}
  \vert \mathcal{R}_3(T)\vert\le
  \frac{\vert B_4\vert}{12(1+T^2)^{3/2}\cos^3(\frac{1}{2}\arg(1+iT))}\le
  \frac{\vert B_4\vert}{12\sqrt{2}T\cos^3(\pi/4 ) }=\frac{1}{180 T},
\end{align}
for $T\ge1$.
In the above computations, $B_2=1/6$ and $B_4=-1/30$ are Bernoulli numbers.

Substituting \eqref{I_1}, \eqref{I_2} and \eqref{I_3} into
\eqref{Integral of C_at 1/2}, and in turn using \eqref{Integral of C_1},
we get the expression
\begin{multline}\label{C_N final}
  \frac{1}{4\pi i}\int\limits_{\partial P(A)}\mathcal{C}_N(s)ds
  =\frac{1}{2}(1-2g)-\frac{\vol(X_N)}{2\pi}\left(\frac{T^2}{2}
  -\frac{1}{24}\right)-\sum_{i=1}^{l}\frac{1}{4m_i}\sum_{j=1}^{m_i-1}
  \frac{1}{\sin^2(\pi j/m_i)} \\
  +\frac{1}{\pi}\left(T\log T-T+\frac{\pi}{4}\right)
  +\frac{\log 2}{\pi}T-\frac{\vol(X_N)}{\pi} g_1(T)
  +\frac{1}{2\pi} g_2(T)+\frac{1}{\pi} g_3(T).
\end{multline}

Using the evaluation \eqref{DefScattViaXi} of the scattering determinant, we
immediately deduce that, inside the rectangle $R_{1/2}(A)$ the function
$\varphi_N (s)$ has a simple pole at $s=1$ and zeros at points $\rho$.
Therefore,
\begin{align}\label{int_phi}
  \frac{1}{4\pi i}\int\limits_{\partial P(A)}
  \frac{\varphi_N'}{\varphi_N}(s)ds
  =Q_N[0<\Im(\rho)\le T]-\frac{1}{2}+\frac{1}{4\pi}\int\limits_{-T}^T
  \frac{\varphi_N'}{\varphi_N}(1/2+it)dt.
\end{align}
Combining \eqref{int_phi} with \eqref{C_N final} and \eqref{sumzeros} yields,
for
$T\ge1$,
\begin{multline}\label{sumzeros 2}
  \mathcal{N}_N[0<r_n\le T]=R_1(T)+\frac{1}{4\pi}
  \int\limits_{-T}^T\frac{\varphi_N'}{ \varphi_N}(1/2+it)dt
  +\frac{\vol(X_N)}{4\pi} T^2-\frac{T\log T}{\pi}
  +\frac{T}{\pi} (1-\log 2) \\
  +\sum_{i=1}^{l}\frac{1}{4m_i}\sum_{j=1}^{m_i-1}
  \frac{1}{\sin^2(\pi j/m_i)}-\frac{\vol(X_N)}{48\pi}-\frac{1}{4}
  -\frac{n_N}{2}-m_{1/4, N}+\left(\frac{\vol(X_N)}{\pi} g_1(T)
  -\frac{1}{2\pi} g_2(T) -\frac{1}{\pi}g_3(T)\right).
\end{multline}

Taking logarithmic derivative of \eqref{DefScattViaXi}, we get
\begin{align}\label{Scat Integral}
  \frac{1}{4\pi}\int\limits_{-T}^T\frac{\varphi_N'}{ \varphi_N}(1/2+it)dt
  =\frac{1}{2\pi}\int\limits_0^T\frac{dt}{(1/4)+t^2}
  -\frac{1}{\pi} \Re\left(-i\int\limits_0^T (\log \xi (1+2it))' dt
  \right)+\frac{1}{4\pi}\int\limits_{-T}^T\frac{D_N'}{ D_N}(1/2+it)dt.
\end{align}
We now will compute the three integrals on the right-hand side of
\eqref{Scat Integral} separately.
First,
\begin{align}\label{J_1}
  \frac{1}{2\pi}\int\limits_0^T\frac{dt}{(1/4)+t^2}=\frac{1}{2}
  -\frac{1}{\pi}\arctan (1/T).
\end{align}
As for the second term on the right-hand side of \eqref{Scat Integral}, we
begin by writing
\begin{align*}
  -\frac{1}{\pi}\Re\left(-i\int\limits_0^T(\log\xi(1+2it))' dt
  \right)=-\frac{1}{\pi}\Im(\log\xi(1+2iT)-\log\xi(1)).
\end{align*}
From the definition of the function $\xi$, one has $\xi (1)=\xi (0)=1/2$, so
then
\begin{align}\label{J_2}
  -\frac{1}{\pi}\Re\left(-i\int\limits_0^T(\log\xi(1+2it))'dt
  \right)&=-\frac{1}{\pi}\Im(\log\xi(1+2iT)-\log\xi(1)) \\ \notag
  &=-1-\frac{T\log T}{\pi}+\frac{T}{\pi}(1+\log\pi)
  -\frac{1}{\pi}g_4(T)-\frac{1}{\pi}R_2(T),
\end{align}
where $R_2(T)=\Im (\log\zeta(1+2iT))=\Im (\log (2iT\zeta(1+2iT)))-\pi/2$.
From Stirling's formula, we have that
\begin{align}\label{g_4}
  g_4(T)=\Im\left(\log\left(1+\frac{1}{2iT}\right)\right)
  +\Re\left(T\log\left(1+\frac{1}{2iT}\right)\right)
  -\frac{2B_2 T}{(1+4T^2)}+\Im(\mathcal{R}_4(T)).
\end{align}
The error term $\mathcal{R}_4(T)$ satisfies the inequality
\begin{align}\label{R_4}
  \vert\mathcal{R}_4(T)\vert\le\frac{\vert B_4\vert}{12 (1/4+T^2)^{3/2}
    \cos^3(\frac{1}{2}\arg(1/2+iT))}\le\frac{2}{225 T},
\end{align}
for all $T\ge1$, which we have deduced in a manner similar to \eqref{R_3}.

As for the third integral in \eqref{Scat Integral}, we begin by noting
that the logarithmic derivative of the function $D_N$ is given by
\begin{align}\label{LogDer D_N}
  \frac{D_N'}{D_N}(s)=-\log N-\sum_{j=1}^r\frac{(p_j-1)\log p_j
    \cdot p_j^s}{(p_j^s+p_j)(p_j^s+1)}.
\end{align}
Furthermore, straightforward computations yield the formula
\begin{align}\label{Part D_N int}
  \frac{1}{2\pi}\int\limits_0^T\Re\left[
    \frac{(p_j-1)p_j^{1/2+it}}{(p_j^{1/2+it}+p_j)(p_j^{1/2+it}+1)}\right]dt
  =\frac{p_j-1}{2\pi}\frac{1}{\log p_j(p_j+1)}
  \int\limits_0^{T\log p_j}\frac{du}{1+a_j\cos u},
\end{align}
where
\begin{align*}
  a_j=2/(\sqrt{p_j}+1/\sqrt{p_j}) = 1/\cosh((1/2)\log(p_{j})).
\end{align*}
With these preliminary computations,
the third term on the right-hand side of \eqref{Scat Integral} can be
evaluated using \eqref{LogDer D_N} and \eqref{Part D_N int}, namely we have
the formula
\begin{align}\label{IntLogDer D_N}
  \frac{1}{4\pi}\int\limits_{-T}^T\frac{D_N'}{D_N}(1/2+it)dt
  =-\frac{\log N}{2\pi}T-\frac{1}{2\pi}\sum_{j=1}^r
  \frac{p_j-1}{p_j+1}\int\limits_0^{T\log p_j}\frac{du}{1+a_j\cos u}.
\end{align}
We write
\begin{align}\label{GRform}
  \int\limits_0^{T\log p_j}\frac{du}{1+a_j\cos u}
  =\sum_{k=0}^{\lfloor\frac{T\log p_j}{\pi}\rfloor-1}
  \int\limits_0^{\pi}\frac{du}{1+a_j\cos u}
  +\int\limits_0^{T\log p_j-\lfloor\frac{T\log p_j}{\pi}\rfloor\pi}
  \frac{du}{1+(-1)^{\lfloor\frac{T\log p_j}{\pi}\rfloor}a_j\cos u}
\end{align}
and use \cite{GR07}, formulas 3.613.1 with $n=0$, $a=a_j$ and 2.553.3 with
$a=1$, $b=(-1)^{\lfloor\frac{T\log p_j}{\pi}\rfloor}a_j$ (hence $b^2<a^2$) to
evaluate the two integrals in \eqref{GRform}.
Substituting \eqref{GRform} into \eqref{IntLogDer D_N}, and employing the
definition
$\alpha_N(j,T):=T\log p_j-\lfloor\frac{T\log p_j}{\pi} \rfloor\pi$,
we get the expression
\begin{align*}
  \frac{1}{4\pi}\int\limits_{-T}^T\frac{D_N'}{D_N}(1/2+it)dt
  =-\frac{\log N}{\pi}T+\frac{1}{2\pi}\sum_{j=1}^r\alpha_N(j,T)
  -\frac{1}{\pi}\sum_{j=1}^r\arctan\left(\left(
  \frac{\sqrt{p_j}-1}{\sqrt{p_j}+1}
  \right)^{(-1)^{\lfloor\frac{T\log p_j}{\pi}\rfloor}}
  \tan\left(\frac{\alpha_N(j,T)}{2}\right)\right).
\end{align*}
Now, by combining this last formula with \eqref{Scat Integral}, \eqref{J_1}
and \eqref{J_2}, we arrive at the expression
\begin{multline}\label{summary1}
  \frac{1}{4\pi}\int\limits_{-T}^T\frac{\varphi_N'}{\varphi_N}(1/2+it)dt
  =-\frac{1}{2}-\frac{T\log T}{\pi}+\frac{T}{\pi}(1+\log(\pi/N) )
  -\frac{1}{\pi}R_2(T)+\frac{1}{2\pi}\sum_{j=1}^r\alpha_N(j,T) \\
  -\frac{1}{\pi}\sum_{j=1}^r\arctan\left(\left(
  \frac{\sqrt{p_j}-1}{\sqrt{p_j}+1}
  \right)^{(-1)^{\lfloor\frac{T\log p_j}{\pi}\rfloor}}
  \tan\left(\frac{\alpha_N(j,T)}{2}\right)\right)
  -\frac{1}{\pi}\left(\arctan(1/T)+g_4(T)\right).
\end{multline}
Substituting \eqref{summary1} into \eqref{sumzeros 2}, we immediately see that
\begin{align*}
  \mathcal{N}_N[0<r_n\le T]-\M_N(T)=S_N(T),
\end{align*}
where
\begin{multline*}
  \M_N(T)=\frac{\vol(X_N)}{4\pi}T^2-\frac{2T\log T}{\pi}
  +\frac{T}{\pi}(2+\log(\pi/2N))
  +\sum_{i=1}^{l}\frac{1}{4m_i}\sum_{j=1}^{m_i-1}
  \frac{1}{\sin^2(\pi j/m_i)}-\frac{\vol(X_N)}{48\pi}-m_{1/4, N} \\
  -\frac{3}{4}-\frac{n_N}{2}+\frac{1}{2\pi}\sum_{j=1}^r\alpha_N(j,T)
  -\frac{1}{\pi}\sum_{j=1}^r\arctan\left(\left(
  \frac{\sqrt{p_j}-1}{\sqrt{p_j}+1}
  \right)^{(-1)^{\lfloor\frac{T\log p_j}{\pi}\rfloor}}
  \tan\left(\frac{\alpha_N(j,T)}{2}\right)\right)+G_N(T),
\end{multline*}
with
\begin{align*}G_N(T)=-\frac{1}{2\pi}\left(-2\vol(X_N)g_1(T)+g_2(T)
  +2g_3(T)+2g_4(T)+2\arctan(1/T)\right)
\end{align*}
and
\begin{align}\label{definition_of_S}
  S_N(T)=R_1(T)-\frac{1}{\pi}(\Im(\log(2iT\zeta(1+2iT)))-\pi/2).
\end{align}

At this time, it remains to derive bounds for the error terms $G_N(T)$ and
$S_N(T)$.
From the definition \eqref{g_1 def} of the function $g_1(T)$ we deduce that
\begin{align*}
  \vert g_1(T)\vert\le\int\limits_T^{\infty}t e^{-2\pi t}dt
  =\frac{e^{-2\pi T}}{2\pi}(T+\frac{1}{2\pi}).
\end{align*}
For an arbitrary positive constant $A>0$, the function
$f(x)=x^2\exp (A-Ax)$ is decreasing for $x > 2/A$; hence, if
$A>2$, then $f(x)\le f(1)=1$ for all $x\ge1$.
Therefore, for $A>2$, one gets $\exp(-Ax)\le\exp(-A)x^{-2}$ for all
$x\ge1$.
Taking $A=2\pi>2$, we obtain the bound
\begin{align}\label{g_1 bound}
  \vert g_1(T)\vert\le\frac{e^{-2\pi T}}{2\pi}(T+\frac{1}{2\pi})\le
  \frac{2\pi+1}{4\pi^2\exp(2\pi)}\cdot\frac{1}{T},
\end{align}
for all $T\ge1$.
Since $u\exp(1-u)\le1$ for all $u>0$, we get the inequalities
\begin{equation}
  \vert g_2(i,j,T)\vert \le\frac{2}{\pi}\int\limits_{\pi T}^{\infty}
  \exp(\left(\vert(2j/m_i)-1\vert-1\right)u)du=\frac{2}{\pi}
  \frac{\exp(\left(\vert(2j/m_i)-1\vert-1\right)\pi T)}{\left(1
    -\vert(2j/m_i)-1\vert\right)}
  \le\frac{2}{\left(1-\vert(2j/m_i)-1\vert\right)^2\pi^2 e}
  \cdot\frac{1}{T}.\notag
\end{equation}
For $j\in\{1,\ldots,m_i-1\}$ one has $1-\vert(2j/m_i)-1\vert\ge2/m_i$,
hence
\begin{equation}\label{g_2 bound}
  \vert g_2(T)\vert\le\sum_{i=1}^{l}\frac{1}{m_i}\sum_{j=1}^{m_i-1}
  \frac{2}{e\cdot\sin(\pi j/m_i)\left(1-\vert(2j/m_i)-1\vert
    \right)^2\pi T}\le\sum_{i=1}^{l}\frac{m_i}{2e\pi}
  \sum_{j=1}^{m_i-1}\frac{1}{\sin(\pi j/m_i)}\cdot\frac{1}{T},
\end{equation}
for all $T>1$.

In order to obtain bounds for $g_3$ and $g_4$, we need to estimate
$\Im\left(\log\left(1+\frac{a}{iT}\right)\right)$ and
$\Re\left(T\log\left(1+\frac{a}{iT}\right)\right)$ for $a=1$ and
$a=1/2$.
When $T>1$ one has $\vert a/iT\vert<1$, so then
\begin{align*}
  \log\left(1+\frac{a}{iT}\right)=\sum_{k=1}^{\infty}
  \frac{(-1)^{k-1}}{k}\left(\frac{a}{iT}\right)^{k}.
\end{align*}
Therefore,
\begin{align*}
  \left|\Im\left(\log\left(1+\frac{a}{iT}\right)\right)\right|
  =\left|\sum_{k=1}^{\infty}\frac{(-1)^{k}}{2k-1}\left(
  \frac{a}{T}\right)^{2k-1}\right|\le\frac{a}{T}.
\end{align*}
Similarly,
\begin{align*}
  \left|\Re\left(T\log\left(1+\frac{a}{iT}\right)\right)\right|
  =\left|T\sum_{k=1}^{\infty}\frac{(-1)^{k-1}}{2k}
  \left(\frac{a}{T}\right)^{2k}\right|\le\frac{a^2}{2T}.
\end{align*}
Now, from \eqref{g_3}, \eqref{R_3}, \eqref{g_4} and \eqref{R_4} we conclude
that for $T>1$
\begin{align}\label{g_3 bound}
  \vert g_3(T)\vert\le\frac{1}{2T}+\frac{1}{2T}+\frac{1}{12T}+\frac{1}{180T}
  =\frac{49}{45T}
\end{align}
and
\begin{align}\label{g_4 bound}
  \vert g_4(T)\vert\le\frac{1}{2T}+\frac{1}{8T}+\frac{1}{12T}
  +\frac{2}{225T}=\frac{1291}{1800T}
\end{align}

Finally, for $T\ge1$ one has $\arctan (1/T)\le1/T$, hence substituting
\eqref{g_1 bound}, \eqref{g_2 bound}, \eqref{g_3 bound} and
\eqref{g_4 bound} into the definition of $G_N(T)$, we arrive at
\begin{align*}
  \vert G_N(T)\vert
  \le\frac{1}{2\pi T}\left(\frac{\vol(X_N)(2\pi+1)}{2\pi^2\exp(2\pi)}
  +\sum_{i=1}^{l}\frac{m_i}{2e\pi}
  \sum_{j=1}^{m_i-1}\frac{1}{\sin(\pi j/m_i)}+\frac{5051}{900}\right),
\end{align*}
which is the inequality stated in \eqref{G_N bound}.

By \eqref{R_1 bound}, the proof of the theorem will be complete once we
show that
\begin{align*}
  \int\limits_0^T(\Im(\log(2it\zeta(1+2it)))-\pi/2)dt
  =O\left(\frac{T}{\log^2 T}\right)\ \text{ as }\ T\to\infty.
\end{align*}
In fact, we will prove the stronger bound
\begin{align}\label{errorbound}
  \int\limits_0^T(\Im(\log(2it\zeta(1+2it)))-\pi/2)dt=O(\log T)
  \ \text{ as }\ T\to\infty.
\end{align}
By the changes of variables $s = 1 + 2it$, we can write
\begin{align*}
  \int\limits_0^T\log(2it\zeta(1+2it))dt
  =\frac{1}{2i}\int\limits_{1}^{1+2iT}\log((s-1)\zeta(s))ds.
\end{align*}
The function $\log((s-1)\zeta(s))$ is holomorphic in the closed rectangle with
vertices $1$, $A$, $A+2iT$ and $1+2iT$, so, by Cauchy's theorem, we have that
\begin{multline}\label{zetaintegral}
  \int\limits_0^T\log(2it\zeta(1+2it)) dt
  =\frac{1}{2i}\int\limits_{1}^{A}\log((\sigma-1)
  \zeta(\sigma))d\sigma+\int\limits_0^T\log((A-1+2it)\zeta(A+2it))dt \\
  +\frac{1}{2i}\int\limits_{A}^{1}\log((\sigma+2iT-1)
  \zeta(\sigma+2iT))d\sigma
  =O(1)+J_1(T)+J_2(T)\ \text{ as }\ T\to\infty.
\end{multline}

It remains to estimate $\Im (J_1 (T))$ and $\Im (J_2 (T))$.
Trivially, one has
\begin{align*}
  \Im (J_1(T))=\Im\left(\int\limits_0^T\log\left(2it\left(
  1+\frac{A-1}{2it}\right)\right)\right)
  +\Im\left(\int\limits_0^T\log\zeta(A+2it)\right).
\end{align*}
It is elementary to show that
$\Im\left(\log\left(2it\left( 1+\frac{A-1}{2it}\right)\right)\right)
=\pi/2+O(t^{-2})$ for $t\gg1$.
The Dirichlet series representation of $\log\zeta(A+2it)$,
is absolutely and uniformly convergent in the range under consideration since
$A>1$.
Therefore, we get the bounds
\begin{align}\label{J1_bound}
  \Im (J_1(T))=\frac{\pi}{2}T+O(1) + \sum_{n=1}^{\infty}
  \frac{\Lambda(n)}{n^{A}\log^2 n} \Im\left(\frac{1-n^{-2iT}}{2i}\right)
  =\frac{\pi}{2}T+O(1)\ \text{ as }\ T\to\infty,
\end{align}
where $\Lambda(n)$ is the von Mangoldt function.
Combining \eqref{J1_bound} with \eqref{zetaintegral}, we have that
\begin{align*}
  \int\limits_0^T(\Im(\log(2it\zeta(1+2it)))-\pi/2)dt
  =\Im (J_2(T))+O(1)\ \text{ as }\ T\to\infty.
\end{align*}
In order to prove \eqref{errorbound}, we need to show that
$\Im (J_2(T))=O(\log T)$ as $T\to\infty$.
The proof of this bound is straightforward.
Simply combine the elementary bound $\log(\sigma+2iT-1)=O(\log T)$ together
with the estimate $\log\zeta(\sigma+2iT)=O(\log T)$ which holds uniformly for
$\sigma\in[1,A]$, which we quote from Theorem 3.5 in \cite{Ti86}.

With all this, the proof of Theorem \ref{AverageWeylLaw} is complete.

\begin{remark}\rm
  In order to prove Corollary \ref{ClassicalWeylLaw}, one follows the
  analysis above up to equation \eqref{definition_of_S}.
  At that point, one uses the first part of Theorem 2.29 on page 468 of
  \cite{Hej83} to bound the first term and Theorem 3.5 in \cite{Ti86} to
  bound the second term.
\end{remark}

We now state three special cases of Theorem \ref{AverageWeylLaw}; first when
$\Gamma = \PSL(2,\Z)$, next when $\Gamma = \Gamma_0(5)^+$, and finally when
$\Gamma = \Gamma_0(6)^+$.

\begin{corollary}[Average Weyl's law for $\PSL(2, \Z)$]
  \begin{align*}
    \mathcal{N}_1[0<r_n\le T]-\M_1(T)=S_1(T),
  \end{align*}
  where
  \begin{align*}
    \M_1(T)=\frac{1}{12} T^2-\frac{2T\log T}{\pi}
    +\frac{T}{\pi}(2+\log(\pi/2))-\frac{131}{144}+G_1(T),
  \end{align*}
  with
  \begin{align*}
    \vert G_1(T)\vert\le\frac{1}{2\pi}\left(
    \frac{2\pi+1+6(1+2\sqrt{3})\exp(2\pi-1)}{6\pi\exp(2\pi)}
    +\frac{5051}{900}\right)\frac{1}{T}<\frac{1}{T}
  \end{align*}
  and
  \begin{align*}
    \int\limits_0^T S_1(t)dt=O\left(\frac{T}{\log^2T}\right),
    \ \text{ as }\ T\to\infty.
  \end{align*}
\end{corollary}
\begin{proof}
  We apply Theorem \ref{AverageWeylLaw} with $N=1$.
  In this case, $D_1(s) \equiv 1$ and the signature of the group is
  $(0; 2,3;1)$.
  Furthermore, $\lambda_1 >1/4$, by Theorem 11.4 from
  \cite{Iw02}, hence $n_1 =1$ and $m_{1/4, N}=0$.
\end{proof}

\begin{remark}\rm
  We have been informed that in \cite{Boo12}, the authors prove an
  average Weyl's law for $\SL(2,\Z)$ together with effective bounds for the
  integral of $S_{1}$, using a trace formula approach.
\end{remark}

\begin{corollary}[Average Weyl's law for $\Gamma_0(5)^+$]
  Let $\alpha_5(T)=T\log 5-\left\lfloor\frac{T\log 5}{\pi}\right\rfloor\pi$.
  Then,
  \begin{align*}
    \mathcal{N}_5[0<r_n\le T]-\M_5(T)=S_5(T),
  \end{align*}
  where
  \begin{multline*}
    \M_5(T)=\frac{T^2}{4}-\frac{2T\log T}{\pi}
    +\frac{T}{\pi} (2+\log\left(\frac{\pi}{10}\right) )-\frac{43}{48}
    +\frac{\alpha_5(T)}{2\pi} \\
    -\frac{1}{\pi}\arctan\left(\left(\frac{\sqrt{5}-1}{\sqrt{5}+1}
    \right)^{(-1)^{\lfloor\frac{T\log 5}{\pi}\rfloor}}
    \tan\left(\frac{\alpha_5(T)}{2}\right)\right)+G_5(T),
  \end{multline*}
  with
  \begin{align*}
    \vert G_5(T)\vert\le\frac{1}{2\pi}\left(
    \frac{2\pi+1+6\exp(2\pi-1)}{2\pi\exp(2\pi)}+\frac{5051}{900}\right)
    \frac{1}{T}<\frac{1}{T}
  \end{align*}
  and
  \begin{align*}
    \int\limits_0^T S_5(t)dt=O\left(\frac{T}{\log^2T}\right),
    \ \text{ as }\ T\to\infty.
  \end{align*}
\end{corollary}
\begin{proof}
  We apply Theorem \ref{AverageWeylLaw} with $N=5$.
  In this case, the signature of the group is
  $(0; 2,2,2 ;1)$.
  Also, the only eigenvalue $\le1/4$ is $\lambda_0=0$,
  by Lemma \ref{MaassFormsLemma} and Corollary 11.5 from \cite{Iw02}.
\end{proof}

\begin{corollary}[Average Weyl's law for $\Gamma_0(6)^+$]
  Let $\alpha_6(T)=T\log 6-\left\lfloor\frac{T\log 2}{\pi}\right\rfloor
  \pi-\left\lfloor\frac{T\log 3}{\pi}\right\rfloor\pi$.
  Then,
  \begin{align*}
    \mathcal{N}_6[0<r_n\le T]-\M_6(T)=S_6(T),
  \end{align*}
  where
  \begin{multline*}
    \M_6(T)=\frac{T^2}{4}-\frac{2T\log T}{\pi}
    +\frac{T}{\pi} (2+\log\left(\frac{\pi}{12}\right) )-\frac{43}{48}
    +\frac{\alpha_6(T)}{2\pi}+G_6(T) \\
    \begin{aligned}
      &-\frac{1}{\pi}\arctan\left(\left(\frac{\sqrt{2}-1}{\sqrt{2}+1}
      \right)^{(-1)^{\lfloor\frac{T\log 2}{\pi}\rfloor}}
      \tan\frac{1}{2}\left( T\log 2
      -\left\lfloor\frac{T\log 2}{\pi}\right\rfloor\pi\right)\right) \\
      &-\frac{1}{\pi}\arctan\left(\left(\frac{\sqrt{3}-1}{\sqrt{3}+1}
      \right)^{(-1)^{\lfloor\frac{T\log 3}{\pi}\rfloor}}
      \tan\frac{1}{2}\left( T\log3
      -\left\lfloor\frac{T\log 3}{\pi}\right\rfloor\pi\right)\right),
    \end{aligned}
  \end{multline*}
  with
  \begin{align*}
    \vert G_6(T)\vert\le\frac{1}{2\pi}\left(
    \frac{2\pi+1+6\exp(2\pi-1)}{2\pi\exp(2\pi)}+\frac{5051}{900}\right)
    \frac{1}{T}<\frac{1}{T}
  \end{align*}
  and
  \begin{align*}
    \int\limits_0^T S_6(t)dt=O\left(\frac{T}{\log^2 T}\right),
    \ \text{ as }\ T\to\infty.
  \end{align*}
\end{corollary}
\begin{proof}
  The proof is a straightforward corollary of Theorem \ref{AverageWeylLaw} and
  basic properties of $\Gamma_0(6)^+$.
\end{proof}

\section{Numerical computations}\label{numerics}

In this section we present numerical results on computations and statistical
distribution of large sets of consecutive eigenvalues of Maass cusp forms on
$X_5$ and $X_6$.

\subsection{Computation of consecutive list of eigenvalues of Maass forms on
  $X_5$ and $X_6$}

A systematic search \cite{The12} for Maass cusp forms on $\Gamma_0(5)^+$ in
the interval $0<\lambda<125^2+1/4$ and on $\Gamma_0(6)^+$ in the interval
$0<\lambda<230^2+1/4$ results in $3557$ and $12474$ Maass forms,
respectively.
A few eigenvalues are listed in Table \ref{eigenvalues}.
At some point, the entire list of eigenvalues will be made publicly available.
Prior to that time, the list will be made available to anyone upon request.

We note that the lowest point of the fundamental domain of the surafce $X_6$
has a larger imaginary part than that for $X_5$.
The height $y$ of the lowest point has an influence on how many terms are
to be considered in the Fourier expansion \eqref{TruncatedExp}.
This is the reason, why the computations were much faster on $X_6$ than on
$X_5$.

\begin{table}
  \begin{tabular}{c|cc}
    $n$ &
    $\lambda_n$ for $\Gamma_0(5)^+$ & $\lambda_n$ for $\Gamma_0(6)^+$ \\
    \hline \\
    1 & 17.32676 & 20.93844 \\
    2 & 24.23291 & 26.24717 \\
    3 & 36.89998 & 37.71537 \\
    4 & 40.58784 & 40.01593 \\
    5 & 46.81219 & 52.39092 \\
    \vdots & \vdots & \vdots \\
    3555 & 15623.315 & 15649.988 \\
    3556 & 15623.860 & 15654.937 \\
    3557 & 15625.094 & 15665.201 \\
    \vdots & \vdots & \vdots \\
    12470 & & 52875.046 \\
    12471 & & 52876.076 \\
    12472 & & 52879.257 \\
    12473 & & 52894.324 \\
    12474 & & 52899.011 \\
    \vdots & & \vdots \\
  \end{tabular}
  \caption{\label{eigenvalues} Eigenvalues of the
    Maass cusp forms on $\Gamma_0(5)^+$ in the interval $0<\lambda<125^2+1/4$
    and on $\Gamma_0(6)^+$ in the interval $0<\lambda<230^2+1/4$.}
\end{table}

The algorithm for computing eigenvalues is described in detail in
\cite{The12}.
The main ingredients are the following.
First, using a set of trial values $\tilde\lambda_1,\ldots,\tilde\lambda_\nu$,
we linearize Hejhal's system of equations \eqref{Hejhal} in the eigenvalue
$\lambda$ around each trial value $\tilde\lambda$.
For each $\tilde\lambda$, we obtain a matrix eigenvalue equation which
is then solved numerically.
In this step, the eigenvalues $\lambda$ of Maass cusp forms are related to
the matrix eigenvalues via perturbation theory.
As a result, one obtains a preliminary list of potential eigenvalues
of Maass cusp forms.
For each potential eigenvalue, we solve \eqref{Hejhal} for $0\not=|m|\le M_0$
and check whether the corresponding non-trivial solution is indeed a Maass
cusp form.

The check is described in section \ref{HejhalAlgorithm}.
This results in a verified list of Maass cusp forms.
Finally, we need to check and verify that the list of Maass cusp forms is
consecutive.
As stated, this check is performed using ''average'' Weyl's law and Turing's
method.
If it turns out that eigenvalues are missing, we search for them, using
additional trial values $\tilde\lambda$, until our list of Maass cusp forms
becomes consecutive, as indicated by Turing's method.

We do not have rigorous Turing bounds, yet.
Therefore, we use Turing's method heuristically.
In light of the data obtained, and presented in various figures in this
section, let us again discuss Turing's method, this time keeping the
figures in mind.

Let $\mathcal{N}^{\text{num}}_N(T)$ count the number of numerically found
eigenvalues in the interval $1/4<\lambda\le T^2+1/4$.
The difference between the number of numerically found eigenvalues and the
average Weyl's law
\begin{align*}
  S^{\text{num}}_N(T):=\mathcal{N}^{\text{num}}_N(T)-\M_N(T)
\end{align*}
is a fluctuating function.
Its mean comes close to a non-positive integer whose absolute value counts the
number of solutions which have been overlooked.

\begin{figure}
  \includegraphics{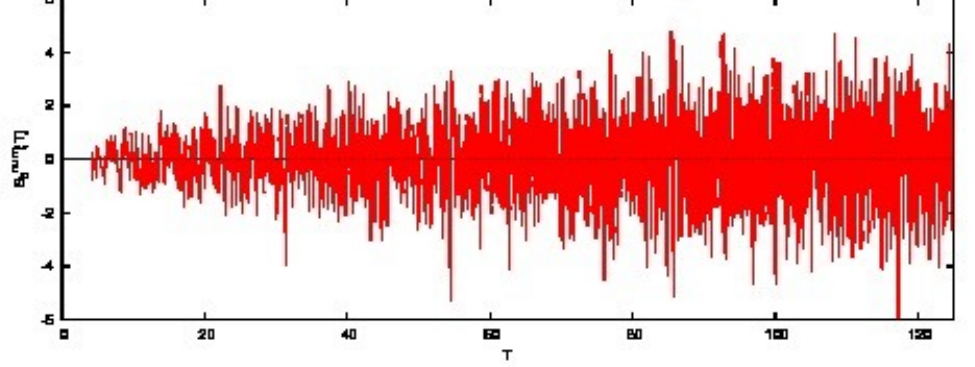}
  \caption{\label{S5}The fluctuations $S^{\text{num}}_5(T)$ for $\Gamma_0(5)^+$.}
\end{figure}

\begin{figure}
  \includegraphics{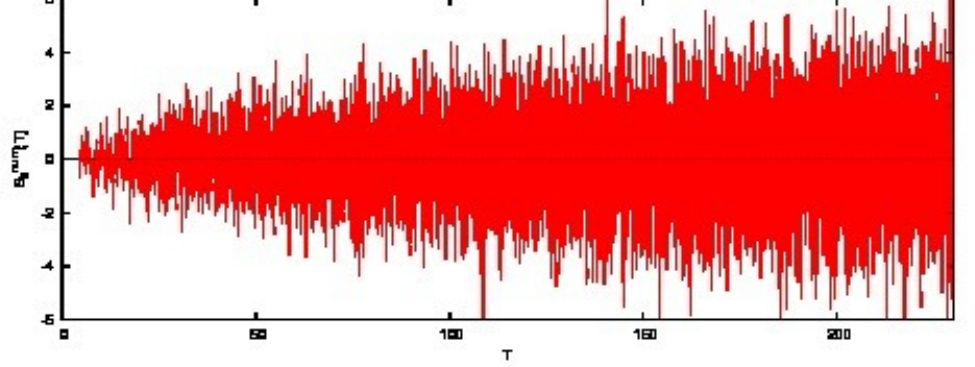}
  \caption{\label{S6}The fluctuations $S^{\text{num}}_6(T)$ for $\Gamma_0(6)^+$.}
\end{figure}

\begin{figure}
  \includegraphics{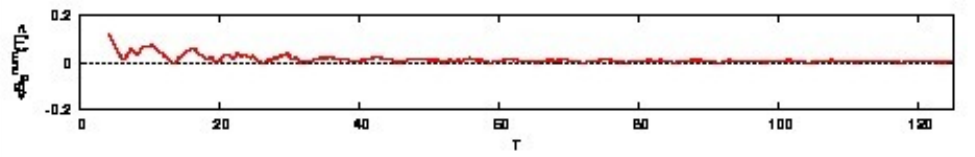}
  \caption{\label{S5m}The mean $\langle S^{\text{num}}_5(T)\rangle$.}
\end{figure}

\begin{figure}
  \includegraphics{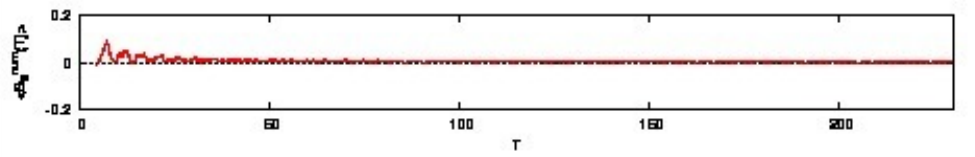}
  \caption{\label{S6m}The mean $\langle S^{\text{num}}_6(T)\rangle$.}
\end{figure}

\begin{figure}
  \includegraphics{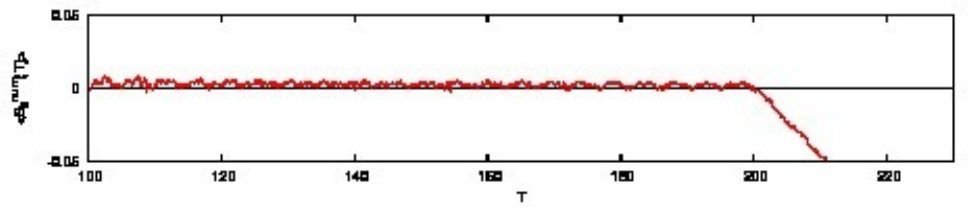}
  \caption{\label{S6m_miss}Mean $\langle S^{\text{num}}_6(T)\rangle$, with the
    eigenvalue $\lambda_{9367}=200.0359^2+1/4$ removed.}
\end{figure}

\begin{figure}
  \includegraphics{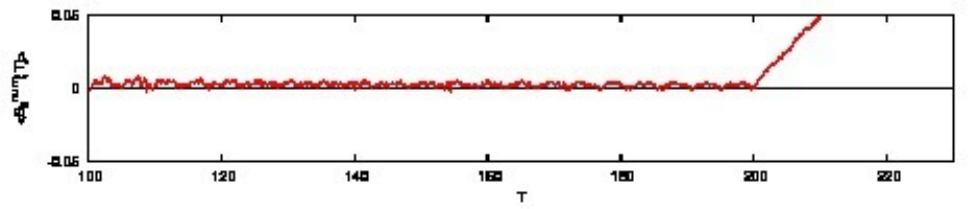}
  \caption{\label{S6m_fake}Mean $\langle S^{\text{num}}_6(T)\rangle$, with the
    fake ``eigenvalue'' $\lambda=200^2+1/4$ inserted.}
\end{figure}

Figures \ref{S5} and \ref{S6} show the fluctuations $S^{\text{num}}_N(T)$.
In figures \ref{S5m} and \ref{S6m}, the mean
\begin{align*}
  \langle S^{\text{num}}_N(T)\rangle
  :=\frac{1}{T}\int\limits_{0}^{T}S^{\text{num}}_N(t)dt
\end{align*}
tends to zero for large $T$ which indicates that all solutions have been found
numerically.

If a solution would have been overlooked, the graph would deviate from zero
quite significantly.
A demonstration is given in figure \ref{S6m_miss}, where we have intentionally
removed the eigenvalue $\lambda_{9367}=200.0359^2+1/4$, whereas in figure
\ref{S6m_fake}, we have intentionally inserted a fake ``eigenvalue'' at
$\lambda=200^2+1/4$.

If we would have an explicit and efficient upper bound on
$\int\limits_{0}^{T} S_N(t) dt$, we could apply Turing's
method to prove, not just verify, that the numerically found lists of
eigenvalues are consecutive.
The proof would be to add a fake ``eigenvalue'' near the end of each list of
eigenvalues and show that with this extra ``eigenvalue''
$\langle S^{\text{num}}_N(T)\rangle$ would exceed the upper bound, as explained
in section \ref{Turing}, see also \cite{Tur53,Boo06,Boo12}.
In our notation, what is needed is to explicitly evaluate the implied constant
in the average of $S_N(T)$.
The algorithm in \cite{FJK10} and \cite{JK12} does, in fact provide such a
bound, but the explicit value is somewhat large, hence impractical.

\begin{remark}\rm
  For computing $S^{\text{num}}_N(T)$ we need to evaluate $\M_N(T)$ which
  includes the term $G_N(T)$.
  Actually, we do not know the exact value of $G_N(T)$.
  According to the bound \eqref{G_N bound}, we can safely neglect $G_N(T)$ in
  the evaluation of $S^{\text{num}}_N(T)$ for $T$ large.

  By Theorem \ref{AverageWeylLaw}, $\M_N(T)$ includes terms which
  depend on $\alpha_N(j,T)$ and on $\arctan(\cdot)$.
  For evaluating the average $\langle S^{\text{num}}_N(T)\rangle$, we need to
  integrate over these terms.
  The sum of the $\alpha_N(j,T)$ and the $\arctan(\cdot)$ dependent terms is
  periodic.
  We perform the integration by expanding the periodic contribution into a
  Fourier series, integrate the individual Fourier terms, and then sum up
  numerically.
\end{remark}

\subsection{Nearest neighbour spacing statistics}\label{spacing}

Concerning the statistical properties of the eigenvalues, we must emphasize
that the conjectured properties depend on the choice of the surface $X$.
Depending on whether the corresponding classical system of a point particle
that moves freely on the surface is integrable or not, there are some
generally accepted conjectures about the nearest neighbour spacing
distributions of the eigenvalues in the limit $\lambda\to\infty$.

Whenever we examine the distribution of the eigenvalues we consider the values
on the scale of the mean level spacings.

\begin{conjecture}[\cite{BT76}]
  If the corresponding classical system is integrable, the eigenvalues
  behave like independent random variables and the distribution of the
  nearest neighbour spacings is in the limit $\lambda\to\infty$ close to a
  Poisson distribution, i.e.\ there is no level repulsion.
\end{conjecture}

\begin{conjecture}[\cite{BGS86}]
  If the corresponding classical system is chaotic, the eigenvalues are
  distributed like the eigenvalues of hermitian random matrices.
  The corresponding ensembles depend only on the symmetries of the system:
  \begin{itemize}
  \item
    For chaotic systems without time-reversal invariance the distribution
    of the eigenvalues approaches in the limit $\lambda\to\infty$ the
    distribution of the Gaussian Unitary Ensemble (GUE) which is characterised
    by a quadratic level repulsion.
  \item
    For chaotic systems with time-reversal invariance and integer spin
    the distribution of the eigenvalues approaches in the limit
    $\lambda\to\infty$ the distribution of the Gaussian Orthogonal Ensemble
    (GOE) which is characterised by a linear level repulsion.
  \item
    For chaotic systems with time-reversal invariance and half-integer
    spin the distribution of the eigenvalues approaches in the limit
    $\lambda\to\infty$ the distribution of the Gaussian Symplectic Ensemble
    (GSE) which is characterised by a quartic level repulsion.
  \end{itemize}
\end{conjecture}
These conjectures are very well confirmed by numerical calculations,
but several exceptions are known.

\begin{exception}\rm
  The harmonic oscillator is classically integrable, but its spectrum
  is equidistant.
\end{exception}

\begin{exception}\rm
  The geodesic motion on surfaces with constant negative curvature
  provides a prime example for classical chaos.
  In some cases, however, the nearest neighbour distribution of the
  eigenvalues of the Laplacian on these surfaces appears to be Poissonian.
\end{exception}

With our lists of consecutive eigenvalues, we can examine the nearest
neighbour spacings.
We unfold the spectrum
\begin{align*}
  u_n=\M_N(r_n)\ \text{ with }\ \lambda_n=r_n^2+1/4,
\end{align*}
in order to obtain rescaled eigenvalues $u_n$ with a unit mean density.
Then
\begin{align*}
  s_n=u_{n+1}-u_n
\end{align*}
defines the sequence of nearest neighbour level spacings which has a mean
value of $1$ as $n\to\infty$.
For the moonshine groups $\Gamma_0(5)^+$ and $\Gamma_0(6)^+$ we find that
the spacing distributions come close to that of a Poisson random process,
\begin{align*}
  P_{\text{Poisson}}(s)=e^{-s},
\end{align*}
see figure \ref{P5&6}, as opposed to that of a Gaussian
orthogonal ensemble of random matrix theory,
\begin{align*}
  P_{\text{GOE}}(s)\simeq\frac{\pi}{2}se^{-\frac{\pi}{2}s^2}.
\end{align*}
The spacing distributions are in accordance with Conjecture \ref{aqc}.

\begin{figure}
  \includegraphics{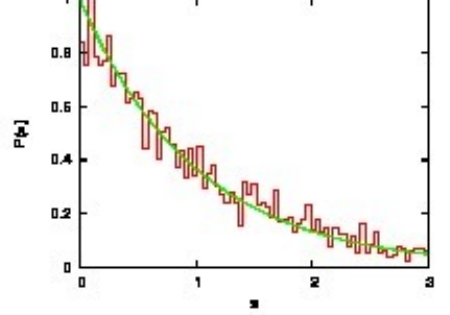} \hfill \includegraphics{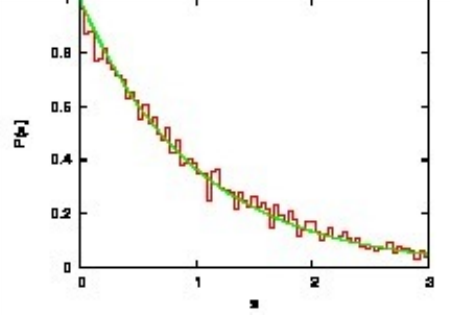}
  \caption{\label{P5&6}Nearest neighbour spacing distributions $P(s)$ for the
    moonshine groups $\Gamma_0(5)^+$ (left), and $\Gamma_0(6)^+$ (right),
    which come close the Poisson distribution $P_{\text{Poisson}}(s)=e^{-s}$.}
\end{figure}

One might wonder whether eigenvalue spacings are correlated.
For this we investigated joint eigenvalue spacing distributions,
\begin{align*}
  P(s,s')dsds'=P\big((s_n,s_{n+1})\in[s,s+ds)\times[s',s'+ds')\big).
\end{align*}
For $\Gamma_0(5)^+$ and $\Gamma_0(6)^+$, we find that the joint eigenvalue
spacing distributions factor into a product of Poisson distributions,
\begin{align*}
  P(s,s')=P_{\text{Poisson}}(s)P_{\text{Poisson}}(s'),
\end{align*}
see figures \ref{Pj5} and \ref{Pj6}.
Heuristically, the spacings between rescaled eigenvalues are uncorrelated
which implies that the eigenvalues are uncorrelated as well.

\begin{figure}
  \includegraphics{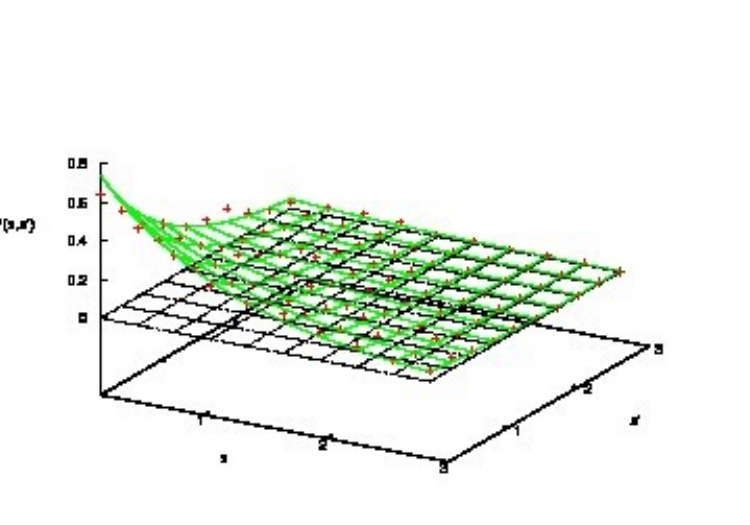}
  \caption{\label{Pj5}Joint nearest neighbour spacing distribution $P(s,s')$
    for $\Gamma_0(5)^+$ which comes close to a product of Poisson
    distributions $P_{\text{Poisson}}(s)P_{\text{Poisson}}(s')$.}
\end{figure}

\begin{figure}
  \includegraphics{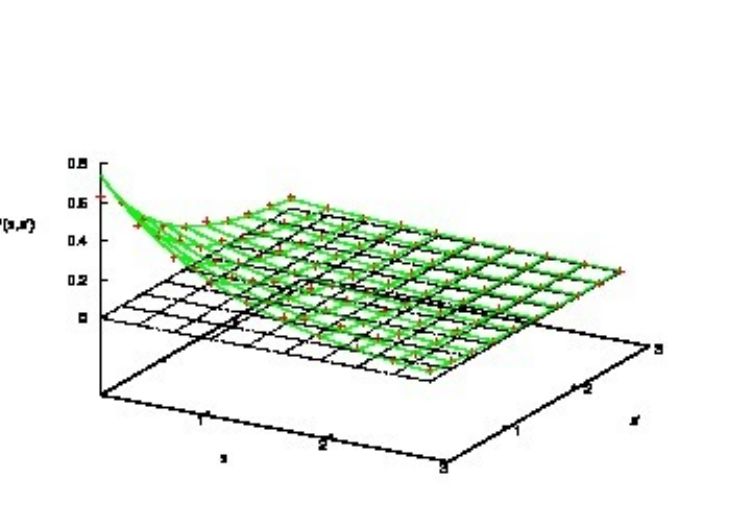}
  \caption{\label{Pj6}Joint nearest neighbour spacing distribution $P(s,s')$
    for $\Gamma_0(6)^+$ which comes close to the product of Poisson
    distributions $P_{\text{Poisson}}(s)P_{\text{Poisson}}(s')$.}
\end{figure}

\section{Concluding remarks}\label{conclusions}

\subsection{Topological equivalence versus Weyl's law}

The groups $\Gamma_0(5)^+$ and $\Gamma_0(6)^+$, which were the main focus of
investigation in our paper, are topologically equivalent, have different
Weyl's law, yet the two sets of eigenvalues seem to have the same spacing
distributions.

From the tables presented in \cite{Cum04}, one can find other examples of
such groups.
In the case when the genus $g$ is zero, we have the following examples of
topologically equivalent groups with different ``classical'' Weyl's laws and
``average'' Weyl's law.

\begin{enumerate}
\item
  For $N\in\{11,14,15\}$, the signature of the surface
  $\overline{\Gamma_0(N)^+}\backslash\h$ is $(0; 2,2,2,2;1)$,
\item
  For $N\in\{17,22,30\}$, the signature of the surface
  $\overline{\Gamma_0(N)^+}\backslash\h$ is $(0; 2,2,2,2,2;1)$,
\item
  For $N\in\{23,33,42\}$, the signature of the surface
  $\overline{\Gamma_0(N)^+}\backslash\h$ is $(0; 2,2,2,2,2,2;1)$,
\item
  For $N\in\{29,38\}$, the signature of the surface
  $\overline{\Gamma_0(N)^+}\backslash\h$ is $(0; 2,2,2,2,2,2,2;1)$,
\item
  For $N\in\{46,51,55,66,70\}$, the signature of the surface
  $\overline{\Gamma_0(N)^+}\backslash\h$ is $(0; 2,2,2,2,2,2,2,2;1)$.
\end{enumerate}

There are also examples of the groups with genus $g=1$, such as
$N\in\{83,123,143,182,195\}$ each of which has signature
$(1;2,2,2,2,2,2,2,2,2,2,2,2;1)$ and whose Weyl's law asymptotics differ in
the $T$ term.

This empirical investigation yields to an interesting question: For a given
positive integer $k$, is it possible to find $k$ topologically equivalent
surfaces arising from moonshine groups having different Weyl's laws?

\subsection{Weyl asymptotics versus nearest neighbour statistics}

Generally speaking, discrete eigenvalues of the Laplacian, or, equivalently,
positive imaginary parts of zeros of the corresponding Selberg zeta function
on the critical line, are increasing sequences of numbers, and the associated
Weyl's law is an approximate counting function of such sequences.
The results in section \ref{spacing} are related to numerical computation of
the nearest neighbour statistics of eigenvalues of Maass cusp forms on
$\overline{\Gamma_0(N)^+}\backslash\h$, for $N=5$ and $N=6$.
We have seen, empirically, that the nearest neighbour statistics for the
eigenvalues of Maass cusp forms seem to be each equal even though the Weyl's
laws are different.
One may argue that the reason for this is that the Weyl's law differs in the
$T$ term, while the first two lead terms are the same in the two cases we
considered.

Therefore, a natural question which arises is to what extent does the nearest
neighbour statistics of an increasing sequences of numbers depend on
its average counting function.
The answer to this question is presented in the following example.

\begin{example}\rm
  Let $\{x_n\}_{n\in \N}$ be an increasing sequence of numbers having
  a mean density of $1$, by which we mean
  \begin{align*}
    \lim_{T\to\infty}\frac{1}{T}\#\{x_n\le T\}=1.
  \end{align*}
  Let $m(t)$ be an increasing function, defined for $t>0$, such that
  $m(0)\ge1/2$.
  (In the Weyl's law case, $m(t)=a_0t^2+a_1t\log t+a_2t+a_3+\ldots$, for some
  positive number $a_0$.)
  Let us define a sequence of numbers $\lambda_n$ by letting
  $\lambda_n:=m^{-1}(x_n-\frac{1}{2})$, where $m^{-1}$ denotes the inverse
  function of $m$.
  Let $\mathcal{N}(t)$ be the counting function
  \begin{align*}
    \mathcal{N}(t):=\#\{\lambda_n\le t\},
  \end{align*}
  and let the Weyl asymptotics $\M(t)$ be a smooth approximation to
  $\mathcal{N}(t)$ such that
  \begin{align*}\label{Weyl}
    \lim_{T\to\infty}\frac{1}{T}\int\limits_0^T(\mathcal{N}(t)-\M(t))dt
    =0.
  \end{align*}
  The unfolded spectrum $\{u_n\}$ is defined by $u_n:=m(\lambda_n)$.
  Trivially, $ u_n=m(\lambda_n)=m(m^{-1}(x_n-\frac{1}{2}))=x_n-\frac{1}{2}$,
  for all $n\in \N$ hence $u_{n+1}-u_n=x_{n+1}-x_n$, so the nearest
  neighbour statistics of the unfolded spectrum $\{u_n\}$ equals the nearest
  neighbour statistics of the initial sequence $\{x_n\}$.

  We are free to distribute the sequence of increasing numbers $\{x_n\}$ such
  that the nearest neighbour statistics of $\{x_n\}$ coincides with our
  favorite distribution of non-negative numbers.
  We are also free to choose the smooth increasing function $m(t)$, and hence
  the Weyl asymptotics arbitrarily.
  Since $\{x_n\}$ and $m(t)$ can be chosen independently of each other, we
  conclude that the nearest neighbour statistics of the unfolded spectrum
  $\{u_n\}$ is completely independent of the Weyl asymptotics.
\end{example}

Therefore, all the analytic results on the Weyl asymptotics are completely
independent of the numerical results on the nearest neighbour statistics.
Neither carries any information of the other, regardless of how many expansion
terms we include in the Weyl asymptotics.
Analytics and numerics complement each other.

\end{document}